\documentclass[sn-mathphys,Numbered]{sn-jnl}% Math and Physical Sciences Reference Style
\usepackage{graphicx}%
\usepackage{multirow}%
\usepackage{amsmath,amssymb,amsfonts}%
\usepackage{amsthm}%
\usepackage{mathrsfs}%
\usepackage[title]{appendix}%
\usepackage{xcolor}%
\usepackage{textcomp}%
\usepackage{manyfoot}%
\usepackage{booktabs}%
\usepackage{algorithm}%
\usepackage{algorithmicx}%
\usepackage{algpseudocode}%
\usepackage{listings}%
\usepackage{comment}
\usepackage{footmisc}
\newtheorem{theorem}{Theorem}%  meant for continuous numbers
\newtheorem{lemma}{Lemma}% 
\newtheorem{remark}{Remark}%

\newtheorem{assumption}{Assumption}

\usepackage{caption} 
\usepackage{import}
\usepackage{bm}
\usepackage{arydshln}
\newcommand{\code}{\texttt}
\usepackage{rotating}
\usepackage{xspace}

\newcommand{\st}{{\mathrm{s.t.}}}
\newcommand{\ie}{\textit{i.e.},\xspace}
\newcommand{\expec}{\mathbb{E}}
\newcommand{\Unif}{\mathrm{Unif}}
\def\ltr{\intercal}
\def\tr{^{\ltr}}
\def\Re{\mathbb{R}}
\newcommand{\Diag}{\mathsf{diag}}
\newcommand{\proj}{\mathrm{proj}}

\def\coul{black}
\def\ncoul{black}

\def\rel{\text{relax}}
\def\kkt{\text{KKT}}
\def\dual{\text{dual}}
\def\optval#1#2{{\color{\ncoul}\vartheta_{#1}^{#2}}}

\def\RndThRH{{\color{\ncoul}RndTh\_RH}}

\def\RH{{\color{\coul}\textbf{RH}}\xspace}

\def\mymark#1{\widetilde{#1}}

\def\vax{x}        % continuous first stage variable
\def\cox{c_{\vax}} % objective coefficient of continuous first stage variable
\def\vaz{z}        % binary first stage variable
\def\coz{c_{\vaz}} % objective coefficient of continuous first stage variable
\def\vay{y}        % first continuous second stage variable
\def\vaw{w}        % second continuous second stage variable
\def\dvapi{\pi}    % dual variable (first two constraints) of the second stage
\def\dvamu{\mu}    % dual variable (on lower bounds) of the second stage
\def\dvath{\theta} % dual variable (on upper bounds) of the second stage

\def\lhsy{{\color{\ncoul}G}}
\def\lhsw{{\color{\ncoul}H}}
\def\rhsB{h}
\def\dim#1{\color{\ncoul}d_{#1}} % shortcut for dimension of vectors
\def\feas{\mathcal{P}}

\def\vadem{D}        % first continuous second stage variable
\def\vasup{S}
\def\vaflo{f}
\def\coq{{\color{\ncoul}v}}
\def\vaq{Q}
\def\conewz{{\color{\ncoul}c}}
\def\dvala{\lambda}
\def\dvaal{\alpha}
\newcommand{\nodes}{{\color{\ncoul}\mathsf{N}}}
\newcommand{\arcs}{{\color{\ncoul}\mathsf{A}}}

\def\reac{{\color{\ncoul}R}}

\def\dvapio{\dvapi_0} % dual variable (first constraint) of the second stage
\def\dvapii{\dvapi_1} % dual variable (second constraint) of the second stage
\def\dvalao{\dvala_0}
\def\dvalai{\dvala_1}
\def\lhsytr{{\color{\ncoul}\lhsy\tr}}
\def\lhsyo{{\color{\ncoul}\lhsy_0}}

\def\lhsyi{{\color{\ncoul}\lhsy_1}}

\def\lhswtr{{\color{\ncoul}\lhsw\tr}}
\def\lhswo{{\color{\ncoul}\lhsw_0}}

\def\lhswi{{\color{\ncoul}\lhsw_1}}

\def\brhsB{{\color{\ncoul}\bm{\rhsB}}}
\def\brhsBtr{{\color{\ncoul}\bm{\rhsB\tr}}}
\def\rhsBo{{\color{\ncoul}\bm{\rhsB_0}}}

\def\rhsBi{{\color{\ncoul}\bm{\rhsB_1}}}

\def\vaqmax{{\color{\ncoul}\vaq_{\max}}}

\def\bvax{{\color{\ncoul}\bm{\vax}}}
\def\bcox{{\color{\ncoul}\bm{\cox}}}
\def\bcoxtr{{\color{\ncoul}\bm{\cox}\tr}}
\def\bvaz{{\color{\ncoul}\bm{\vaz}}}
\def\bvazdem{{\color{\ncoul}\bvaz^{\bvadem}}}
\def\bvazsup{{\color{\ncoul}\bvaz^{\bvasup}}}
\def\bvazflo{{\color{\ncoul}\bvaz^{\bvaflo}}}
\def\bcoz{{\color{\ncoul}\bm{\coz}}}
\def\bcoztr{{\color{\ncoul}\bm{\coz}\tr}}
\def\bvay{{\color{\ncoul}\bm{\vay}}}
\def\bvaw{{\color{\ncoul}\bm{\vaw}}}
\def\bdvapi{{\color{\ncoul}\bm{\dvapi}}}
\def\markbdvapi{{\color{\ncoul}\bm{\mymark{\dvapi}}}}
\def\bdvapio{{\color{\ncoul}\bm{\dvapio}}}
\def\bdvapiotr{{\color{\ncoul}\bm{\dvapio\tr}}}

\def\markbdvapiotr{{\color{\ncoul}\bm{\mymark{\dvapi}_0\tr}}}
\def\bdvapii{{\color{\ncoul}\bm{\dvapii}}}

\def\bdvaal{{\color{\ncoul}\bm{\dvaal}}}
\def\bdvala{{\color{\ncoul}\bm{\dvala}}}
\def\bdvalao{{\color{\ncoul}\bm{\dvalao}}}
\def\bdvalaotr{{\color{\ncoul}\bm{\dvalao}\tr}}
\def\bdvalai{{\color{\ncoul}\bm{\dvalai}}}
\def\bdvamu{{\color{\ncoul}\bm{\dvamu}}}
\def\bdvamuy{{\color{\ncoul}\bm{\dvamu^{\bvay}}}}
\def\bdvamuyind#1{{\color{\ncoul}\bm{\dvamu^{\bvay}_{#1}}}}
\def\bdvamuw{{\color{\ncoul}\bm{\dvamu^{\bvaw}}}}
\def\bdvamuwind#1{{\color{\ncoul}\bm{\dvamu^{\bvaw}_{#1}}}}
\def\bdvamuflo{{\color{\ncoul}\bm{\dvamu^{\bvaflo}}}}
\def\markbdvamuflo{{\color{\ncoul}\bm{\mymark{\dvamu}^{\bvaflo}}}}
\def\bdvamudem{{\color{\ncoul}\bm{\dvamu^{\bvadem}}}}
\def\markbdvamudem{{\color{\ncoul}\bm{\mymark{\dvamu}^{\bvadem}}}}
\def\bdvamusup{{\color{\ncoul}\bm{\dvamu^{\bvasup}}}}
\def\markbdvamusup{{\color{\ncoul}\bm{\mymark{\dvamu}^{\bvasup}}}}
\def\bdvath{{\color{\ncoul}\bm{\dvath}}}
\def\bdvathy{{\color{\ncoul}\bm{\dvath^{\bvay}}}}
\def\bdvathyind#1{{\color{\ncoul}\bm{\dvath^{\bvay}_{#1}}}}
\def\bdvathw{{\color{\ncoul}\bm{\dvath^{\bvaw}}}}
\def\bdvathwind#1{{\color{\ncoul}\bm{\dvath^{\bvaw}_{#1}}}}
\def\bdvathflo{{\color{\ncoul}\bm{\dvath^{\bvaflo}}}}
\def\markbdvathflo{{\color{\ncoul}\bm{\mymark{\dvath}^{\bvaflo}}}}

\def\bdvathsup{{\color{\ncoul}\bm{\dvath^{\bvasup}}}}
\def\markbdvathsup{{\color{\ncoul}\bm{\mymark{\dvath}^{\bvasup}}}}

\def\bvadem{{\color{\ncoul}\bm{\vadem}}}
\def\bvasup{{\color{\ncoul}\bm{\vasup}}}
\def\bvaflo{{\color{\ncoul}\bm{\vaflo}}}
\def\bvaq{{\color{\ncoul}\bm{\vaq}}}

\def\ubvaz{{\color{\ncoul}\bm{\overline{\bvaz}}}}
\def\ubvax{{\color{\ncoul}\bm{\overline{\bvax}}}}
\def\ubvaw{{\color{\ncoul}\bm{\overline{\bvaw}}}}
\def\ubvawtr{{\color{\ncoul}\bm{\overline{\bvaw}\tr}}}
\def\ubvay{{\color{\ncoul}\bm{\overline{\bvay}}}}
\def\ubvaytr{{\color{\ncoul}\bm{\overline{\bvay}\tr}}}
\def\ubvaflo{{\color{\ncoul}\bm{\overline{\bvaflo}}}}
\def\ubvaflotr{{\color{\ncoul}\bm{\overline{\bvaflo}\tr}}}
\def\ubvasup{{\color{\ncoul}\bm{\overline{\bvasup}}}}
\def\ubvasuptr{{\color{\ncoul}\bm{\overline{\bvasup}\tr}}}
\def\ubvaq{{\color{\ncoul}\bm{\overline{\bvaq}}}}

\newcommand{\nodesD}{\nodes_{\bvadem}}
\newcommand{\nodesS}{\nodes_{\bvasup}}
\newcommand{\nodeso}{\nodes_{0}}
\newcommand{\nodesi}{\nodes_{1}}

\def\idy#1#2{{\color{\ncoul}\mathbb{I}_{#1}^{#2}}}
\def\idyof#1{\idy{0\bvaflo}{#1}}
\def\idyif#1{\idy{1\bvaflo}{#1}}
\def\idyf#1{\idy{\bvaflo}{#1}}
\def\idyoD#1{\idy{0\bvadem}{#1}}
\def\idyiD#1{\idy{1\bvadem}{#1}}
\def\idyD#1{\idy{\bvadem}{#1}}
\def\idyoS#1{\idy{0\bvasup}{#1}}
\def\idyiS#1{\idy{1\bvasup}{#1}}
\def\idyS#1{\idy{\bvasup}{#1}}

\def\myalpha#1#2{{\color{\ncoul}\alpha_{#1}^{#2}}}
\def\mybeta#1#2{{\color{\ncoul}\beta_{#1}^{#2}}}
\def\mygamma#1#2{{\color{\ncoul}\gamma_{#1}^{#2}}}
\def\mybalpha{{\color{\ncoul}\bm{\myalpha{}{}}}}
\def\mybbeta{{\color{\ncoul}\bm{\mybeta{}{}}}}
\def\mybgamma{{\color{\ncoul}\bm{\mygamma{}{}}}}
\def\myPhi{{\color{\ncoul}\Phi}}
\def\myphi{{\color{\ncoul}\phi}}
\def\mypsi#1#2{{\color{\ncoul}\psi_{#1}^{#2}}}
\def\mybpsi{{\color{\ncoul}\bm{\mypsi{}{}}}}
\def\myOmega{{\color{\ncoul}\Omega}}

\def\nsample{{K}}
\def\samp{{\color{\ncoul}\bm{\xi}}}

\def\mylambda#1#2{{\color{\coul}\lambda_{#1}^{#2}}}

\def\markmyblambda{{\color{\coul}\bm{\mymark{\mylambda{}{}}}}}

\newcommand{\setUnbd}{\infty}
\newcommand{\setBd}{\leftrightarrow}

\usepackage[normalem]{ulem}

\begin{document}

\title[Article Title]{Globally Solving a Class of Bilevel Programs with Spatial Price Equilibrium Constraints}


\author*[1]{\fnm{Akshit} \sur{Goyal}}\email{goyal080@umn.edu}

\author[1]{\fnm{Jean-Philippe P.} \sur{Richard}}\email{jrichar@umn.edu}

\affil*[1]{\orgdiv{Department of Industrial and Systems Engineering}, \orgname{University of Minnesota}, \orgaddress{%\street{Street}, 
\city{Minneapolis}, \state{MN}, \postcode{55455}, \country{USA}}}

\abstract{
Bilevel programs with spatial price equilibrium constraints are strategic models that consider a price competition at the {lower level}.
These models find application in facility location-price models, optimal bidding in power networks, and integration of renewable energy sources in distribution networks. 
In this paper, for the case where the equilibrium at the lower level can be formulated as an optimization problem, we introduce an enhanced single-level formulation based on duality and show that its relaxation is stronger than the single-level formulation obtained using KKT conditions. 
Compared to the literature \cite{friesz1988algorithms,miller1992heuristic}, this new formulation ($i$) is computationally friendly to global solution strategies using branch-and-bound, and ($ii$) can tackle instances of larger size.
Further, we develop a heuristic procedure to find feasible solutions inside of the branch-and-bound tree that is effective on instances of large size and produces solutions whose objective values are close to the relaxation bound.
We demonstrate the benefits of this formulation and heuristic through an extensive numerical study on synthetic instances of Equilibrium Facility Location \cite{miller1995equilibrium} and on standard IEEE bus networks for planning renewable generation capacity under uncertainty. 
}




\maketitle

\section{Introduction}\label{sec1}
Bilevel programs incorporating spatial price equilibrium (SPE) constraints at the lower level are used to model competitive facility location on networks \cite{tobin1986spatial, friesz1988algorithms, miller1995equilibrium} and analyzing bidding decision of a generating firm in an electric power network \cite{hobbs2000strategic}. 
At the core of these bilevel programs lies the concept of SPE, which involves computing the supply price, demand price, and commodity flow in a network while satisfying the equilibrium condition that the demand price equals the supply price plus the transportation cost if there is a non-zero flow between a pair of demand and supply nodes.
In the literature, the general SPE problem has been formulated as a variational inequality (VI) problem \cite{florian1982new} and several VI-based iterative solution procedures, such as Frank-Wolfe and projection methods, have been proposed \cite{friesz1984alternative, nagurney1987computational}. 
Algorithms based on complementarity formulations of the SPE problem have also been developed \cite{pang1981hybrid, pang1981parametric}. 
We refer to \cite{nagurney2021networks} for a comprehensive review of  spatial price equilibria in networks.

The models we study in this paper utilize SPE as the lower-level problem within a bilevel program. 
This approach is necessary to model how the upper-level decisions affect the equilibrium market price of a commodity, taking into account market competition. 
The resulting equilibrium price then directly impacts the objective function at the upper-level.
In \cite{tobin1986spatial}, the authors formulate a bilevel model of this type to locate a firm's production facilities and to determine production decisions at the upper-level in order to maximize the firm's profit,  which depends on the prices arising from the resulting equilibrium at the {lower level}.
In \cite{friesz1988algorithms}, the authors argue the existence of solutions to this equilibrium facility location (EFL) model whereas \cite{tobin1986spatial, friesz1988algorithms} provide heuristic approaches which involve successive linearization of the nonlinear upper-level objective based on the sensitivity analysis results for VIs discussed in  \cite{tobin1986sensitivity}. 
The authors of \cite{miller1992heuristic} extend the work of \cite{tobin1986spatial} by allowing for additional shipping decisions to be made at the upper-level. 
These articles, however, focus on heuristic solution methods that do not provide guarantees on the quality of solutions obtained. 
Further, these heuristics are tested on problem instances of small size. 

In this paper, we present an approach to obtain globally optimal solutions to a class of bilevel programs with SPE constraints which encompasses the EFL application. 
This paper makes the following contributions:
\begin{enumerate}
    \item 
    We derive a new Duality-based single-level reformulation for bilevel programs with SPE constraints %at the lower-level 
    that is stronger than {the KKT-based} reformulation when the variational inequality of the {lower level} can be cast as an optimization problem. 
    This {new} reformulation has a provably bounded root node relaxation, which is advantageous for branch-and-bound. 
    This is in contrast with the {KKT-based} {formulation}, which often has an unbounded root relaxation as we argue theoretically and show computationally in Sections~\ref{sec:SingleLevel} and \ref{sec:NumericExp}, respectively.
    \item 
    We conduct extensive numerical experiments on randomly generated instances of EFL \cite{tobin1986spatial} of varying sizes. 
    To the best of our knowledge, this is the first extensive computational study for this class of bilevel programs. This study establishes that the Duality-based {formulation} allows substantial computational speed ups. 
    \item 
    We introduce an application of bilevel SPE models that provides a novel approach to optimize renewable generation units planning (RGUP) in power distribution networks under uncertainty, extending 
    the work of 
    \cite{zhang2017optimal}.
    We perform numerical experiments on standard IEEE bus networks that show the numerical potential of the approach. 
    \item 
    We develop a generic rounding heuristic procedure for these applications. 
    For larger-sized instances, this heuristic helps solve most instances within an optimality gap of less than $1\%$ in a reasonable amount of time using branch-and-bound solvers.
\end{enumerate}

In Section~\ref{sec:problem_setup}, we introduce the problem, notations, and assumptions. 
We derive a single-level formulation based on KKT conditions in Section~\ref{sec:SingleLevel_Reform1} that we further reformulate using Lagrangian duality in Section~\ref{sec:SingleLevel_Reform2}. 
The theoretical properties of relaxations of these two formulations are discussed in Section~\ref{sec:KKTvsDUAL}.
Lastly, in Section~\ref{sec:NumericExp}, we conduct an extensive computational study on the two applications described above.

%%%%%%%%%%%%%%%%%%%%%%%%%%%%%%%%%%%%%%%%%%%%%%%%%%%%%%%%%%%%%%%%%%%%
\section{Problem description and preliminaries}
\label{sec:problem_setup}
%%%%%%%%%%%%%%%%%%%%%%%%%%%%%%%%%%%%%%%%%%%%%%%%%%%%%%%%%%%%%%%%%%%%
We study SPE-constrained bilevel models where the leader problem is
\begin{subequations} \label{eq:main}
    \begin{align}
        \max_{\bvaz,\bvax}\ &  \bdvapiotr \bvax - \bcoxtr \bvax - \bcoztr   \bvaz \\
        \st\ & {(\bvaz,\bvax) \in \feas} \\
             & {0 \leq \begin{pmatrix} \bvaz \\ \bvax \end{pmatrix} \leq \begin{pmatrix} \ubvaz \\ \ubvax \end{pmatrix}} \\
             & \bdvapi = \begin{pmatrix} \bdvapio \\ \bdvapii \end{pmatrix} : = \mybpsi(\bvax) {\geq 0},
             \label{eq:EqbmPrice}
    \end{align}
\end{subequations}
in which $\bvaz\in\mathcal{Z}:=\mathbb{Z}^{\dim{\bvaz}^I}\times \Re^{\dim{\bvaz}-\dim{\bvaz}^I}$ and $\bvax\in\mathcal{X}:=\mathbb{Z}^{\dim{\bvax}^I}\times \Re^{\dim{\bvax}-\dim{\bvax}^I}$ are the mixed-integer decisions of the leader, $\overline \feas \subseteq \Re^{\dim{\bvaz}+\dim{\bvax}}$ is a closed convex set, 
$\feas = \overline \feas \cap (\mathcal{Z}\times\mathcal{X})$, and $\ubvaz$ and $\ubvax$ are parameters assumed to be finite. 
Further, $\bdvapi$ is an equilibrium price vector constrained to be non-negative and defined as an implicit function $\mybpsi(\cdot)$ of the leader's mixed-integer decision $\bvax$.
Specifically, $\bdvapi$ is a dual solution corresponding to the equality constraints of a variational inequality (whose feasible region depends on $\bvax$) which we refer to as the \textit{follower problem}. 
In case there are multiple dual solutions $\bdvapi$ for a given $\bvax$, one can always encode an optimistic or pessimistic behavior of the follower in the function $\mybpsi(\cdot)$. 
Therefore, we assume that the map $\mybpsi(\cdot)$ is single-valued to ensure that Problem \eqref{eq:main} is well-defined.

For a given $\bvax$, the follower problem is to determine a vector $\bvay^*$ such that 
\begin{align}
   \mathrm{VI}(\myPhi,\mathcal{Y}( \bvax)) : \quad & \langle \myPhi(\bvay^*),\bvay'-\bvay^*\rangle \geq 0 \;\; \forall \bvay' \in \proj_\bvay\mathcal{Y}(\bvax),
   \label{eq:VI-follower}
\end{align} 
where $\myPhi(\bvay)=\big( \myPhi_i(\bvay),\; i\in [\dim{\bvay}]\big)\in\Re^{\dim{\bvay}}$ is the cost vector of the follower and where we use the notation $[\dim{\bvay}]$ for the set $\left\{1,\hdots,\dim{\bvay}\right\}$. 
In \eqref{eq:VI-follower}, $\mathcal{Y}( \bvax)$ is the set of feasible follower solutions, which we assume takes the form
\begin{eqnarray}
\mathcal{Y}(\bvax) = \left\{
\begin{array}{l|l}
(\bvay,\bvaw)\in\Re^{\dim{\bvay}+\dim{\bvaw}} &
   \begin{array}{lll}
  & \lhsyo \bvay + \lhswo \bvaw - \bvax = \rhsBo, \;\;& [{\bdvapio}] \\
  & \lhsyi \bvay + \lhswi \bvaw = \rhsBi, \;\; &[{\bdvapii}]\\
  & \bvay \geq 0, \;\; &[\bdvamuy] \\
  & {\bvay\leq \ubvay}, \;\; &[\bdvathy] \\
  & \bvaw \geq 0, \;\; &[\bdvamuw] \\
  & {\bvaw\leq \ubvaw}, \;\; &[\bdvathw]
  \end{array}
  \end{array}
    \right\}. \label{eq:varfeasibleset} 
\end{eqnarray}
In particular, $\mathcal{Y}( \bvax)$ has continuous variables related by linear constraints, some of which involve the leader mixed-integer decisions $\bvax$. It is admissible for the parameters $\ubvay$ and $\ubvaw$ to be infinite. 
Further, the symbols written inside square brackets in \eqref{eq:varfeasibleset} represent the dual variables associated with the corresponding constraints.
{
\begin{remark}
    Constraint~\eqref{eq:EqbmPrice} requires equilibrium price $\bdvapi$ to be non-negative, which is more meaningful in the applications we study in Section~\ref{sec:NumericExp}. 
    This requirement may lead to infeasibility, which would be detected by solving the single-level formulations of the model presented in Section~\ref{sec:SingleLevel}. 
    However, we do not encounter infeasiblity issues in the  computational experiments reported in Section~\ref{sec:NumericExp}.
\end{remark}
} 
Throughout the paper, we make the following assumptions:

\begin{assumption} \label{assump:JacobSymm}
{The} {vector function $\myPhi$ is continuously differentiable and its Jacobian matrix $\nabla_{\bvay} \myPhi$ is symmetric}.
\end{assumption}

\begin{assumption} \label{assump:Unique} 
For all $\bvay^1$, $\bvay^2 \in\mathrm{dom}\left(\myPhi\right)$ such that $\bvay^1 \neq \bvay^2$, it holds that
{$\left\langle  \myPhi(\bvay^1) -  \myPhi(\bvay^2),\ \bvay^1-\bvay^2 \right\rangle > 0$.
We refer to this property as \textit{strict monotonicity}}.
\end{assumption}

{Let $\bvay_{-i} := (\bvay_{i'}, i'\in [\dim{\bvay}]\setminus\{i\})$.} Under Assumptions \ref{assump:JacobSymm} and \ref{assump:Unique}, Problem~\eqref{eq:VI-follower} can be written as the following convex optimization problem
\begin{align}\label{eq:OPT-follower}
  \min_{(\bvay,\bvaw)\in\mathcal{Y}(\bvax)} \;\; & \myphi_\text{p}(\bvay) := \sum_{i\in [\dim{\bvay}]} \int_{0}^{\bvay_i}  \myPhi_i(\bvay_{-i},u_i)\ du_i
  {,}
\end{align}
see \cite{nagurney2009network} for instance, with {$\nabla_{\bvay} \myphi_\text{p}(\bvay) = \myPhi(\bvay)$}.

\begin{assumption} \label{assump:FollowerFeas}
For any feasible solution $\bvax$ to the leader, the solution set of the follower problem is non-empty, \ie $$\mathcal{Y}^*(\bvax)=\left\{(\bvay^*,\bvaw^*)\in\mathcal{Y}(\bvax)\ \big|\ \langle \myPhi(\bvay^*),\bvay'-\bvay^*\rangle \geq 0 \;\forall \bvay' \in \proj_\bvay\mathcal{Y}(\bvax)\right\}\neq\emptyset.$$ 
\end{assumption}

\begin{remark}\label{remark:CvxFollower}
    Assumption~\ref{assump:Unique} implies that the objective $\myphi_\text{p}(\bvay)$ of the follower is strictly convex. 
    In conjunction with Assumption~\ref{assump:FollowerFeas}, this implies that for every leader decision $\bvax$, there exists a unique {primal} optimal solution $\bvay^*$ to Problem~\eqref{eq:OPT-follower}.
\end{remark}

\begin{remark}\label{remark:RefinedSlaterCondition} 
All equality and inequality constraints in $\mathcal{Y}(\bvax)$ are linear. 
Further, Assumption \ref{assump:FollowerFeas} ensures feasibility of {the} follower problem.  
This implies that Slater’s condition holds \cite[pg.227]{boyd2004convex}. 
Since the follower problem is convex (attaining a noninfinite primal optimal value from Remark~\ref{remark:CvxFollower}), Slater’s condition further implies that ($i$) strong duality holds, and
($ii$) 
%the dual optimal value is attained. %\ie 
a {(not necessarily unique)} dual optimal solution exists.
 %; see \cite[pg.227]{boyd2004convex}. 
\end{remark}

%%%%%%%%%%%%%%%%%%%%%%%%%%%%%%%%%%%%%%%%%%%%%%%%%%%%%%%%%%%%%%%%%%%%
\section{Single-level reformulations}
\label{sec:SingleLevel}
%%%%%%%%%%%%%%%%%%%%%%%%%%%%%%%%%%%%%%%%%%%%%%%%%%%%%%%%%%%%%%%%%%%%

% In this section, 
We describe two formulations for \eqref{eq:main} %the problem 
{when the follower is optimistic}; {see \cite{dempe2002foundations} for differences between optimistic and pessimistic bilevel programs}. 
The first, described in Section~\ref{sec:SingleLevel_Reform1}, is obtained by reformulating the equilibrium problem using the KKT conditions of its equivalent optimization formulation. 
The second, described in Section~\ref{sec:SingleLevel_Reform2}, which is new to this work, is obtained by using duality to rewrite the objective function in the first reformulation using strong duality of {the} follower problem. 
Finally, in Section \ref{sec:KKTvsDUAL}, we compare the strength of the relaxations of these two formulations and develop %additional 
insights into the 
special 
case where the VI cost vector $\myPhi$ is affine.

%%%%%%%%%%%%%%%%%
\subsection{KKT-based {formulation}}
\label{sec:SingleLevel_Reform1}
%%%%%%%%%%%%%%%%%

To obtain the first formulation, we replace \eqref{eq:EqbmPrice} in \eqref{eq:main} with the KKT conditions of Problem~\eqref{eq:OPT-follower}.
We obtain the following MINLP formulation which has complementarity constraints and a bilinear objective function, 
\begin{subequations} \label{eq:SL-MPCC}
\begin{align}
\optval{\kkt}{}
= \max_{\substack{\bvaz,\bvax,\bvay,\bvaw\\ \bdvapi, \bdvamu,\bdvath}}\;\; &  \bdvapiotr \bvax - \bcoxtr \bvax - \bcoztr  \bvaz \label{constr:SL-MPCC-Obj}\\
\st\ \label{constr:SL-MPCC-Leader1} & {(\bvaz,\bvax) \in \feas} \\
     \label{constr:SL-MPCC-Leader2} & {0 \leq \begin{pmatrix} \bvaz \\ \bvax \end{pmatrix} \leq \begin{pmatrix} \ubvaz \\ \ubvax \end{pmatrix}} \\
     \label{constr:SL-MPCC-Leader3} & {\bdvapi \geq 0} \\
     \label{constr:SL-MPCC-PF1} & \lhsy  \bvay + \lhsw \bvaw - \begin{pmatrix} \bvax \\ 0 \end{pmatrix} = \brhsB \\ 
     \label{constr:SL-MPCC-PF2} & 0 \leq \bvay \leq \ubvay,\;\; 0 \leq \bvaw \leq \ubvaw \\
     \label{constr:SL-MPCC-DF} & \bdvamu\geq0, \; \bdvath \geq 0 \\
     \label{constr:SL-MPCC-Stnry1} & \myPhi(\bvay) + \lhsytr \bdvapi + \bdvathy - \bdvamuy = 0 \\ 
     \label{constr:SL-MPCC-Stnry2} & \lhswtr \bdvapi + \bdvathw - \bdvamuw = 0  \\
     \label{constr:SL-MPCC-CC1} & \bvay\tr \bdvamuy = 0,\;\; (\ubvay-\bvay)\tr \bdvathy = 0 \\ 
     \label{constr:SL-MPCC-CC2} & \bvaw\tr \bdvamuw= 0,\;\; (\ubvaw-\bvaw)\tr \bdvathw = 0, 
\end{align}
\end{subequations}
where 
% $\lhsy=\begin{pmatrix} \lhsyo \\ \lhsyi \end{pmatrix}$,\; $\lhsw=\begin{pmatrix} \lhswo \\ \lhswi \end{pmatrix}$,\; $\brhsB=\begin{pmatrix} \rhsBo \\ \rhsBi \end{pmatrix}$, and $\bdvapi=\begin{pmatrix} \bdvapio \\ \bdvapii \end{pmatrix}$.
$${\lhsy=\begin{pmatrix} \lhsyo \\ \lhsyi \end{pmatrix},\; \lhsw=\begin{pmatrix} \lhswo \\ \lhswi \end{pmatrix},\; \brhsB=\begin{pmatrix} \rhsBo \\ \rhsBi \end{pmatrix},\; \text{and}\; \bdvapi=\begin{pmatrix} \bdvapio \\ \bdvapii \end{pmatrix}.}$$
In this formulation, constraints \eqref{constr:SL-MPCC-PF1}-\eqref{constr:SL-MPCC-PF2} are the primal feasibility conditions of the KKT system of Problem~\eqref{eq:OPT-follower}, whereas 
\eqref{constr:SL-MPCC-DF} are its dual feasibility conditions, \eqref{constr:SL-MPCC-Stnry1}-\eqref{constr:SL-MPCC-Stnry2} are its stationary conditions, and \eqref{constr:SL-MPCC-CC1}-\eqref{constr:SL-MPCC-CC2} are its complementarity slackness conditions. 
This reformulation is exact as KKT conditions are necessary and sufficient for Problem~\eqref{eq:OPT-follower} as it is a convex optimization problem that satisfies {Slater’s condition; see Remark \ref{remark:RefinedSlaterCondition}}.

The complementarity constraints \eqref{constr:SL-MPCC-CC1}-\eqref{constr:SL-MPCC-CC2} can be reformulated as big-$M$ constraints. 
This approach, however, can lead to sub-optimal or erroneous solutions when the choice of $M$ is not appropriate, as discussed in \cite{kleinert2023there}. 
Instead, when solving this model with commercial software, we use the SOS1 reformulation of \eqref{constr:SL-MPCC-CC1}-\eqref{constr:SL-MPCC-CC2}:
\begin{align}\label{eq:SOS1}
    \begin{aligned}
    & \{\bvay_i,\bdvamuyind{i}\}\; \text{ is\; SOS1},\;\; \{\ubvay_i-\bvay_i,\bdvathyind{i}\}\; \text{ is\; SOS1}\quad  \forall i \in [\dim{\bvay}] \\
    & \{\bvaw_i,\bdvamuwind{i}\}\; \text{ is\; SOS1},\;\; \{\ubvaw_i-\bvaw_i,\bdvathwind{i}\}\; \text{ is\; SOS1}\quad  \forall i \in [\dim{\bvaw}]
    \end{aligned}
\end{align}
as is recommended in \cite{kleinert2023there}.

{
The formulation \eqref{eq:SL-MPCC} is a non-convex model due to the bilinear term in the objective \eqref{constr:SL-MPCC-Obj}, the complementarity constraints \eqref{constr:SL-MPCC-CC1}-\eqref{constr:SL-MPCC-CC2} and the integrality constraints in \eqref{constr:SL-MPCC-Leader1}. 
Since variables $\bdvamu$, $\bdvath$, and maybe some among variables $\bvay$ are unbounded, it is difficult to use McCormick envelopes to relax complementarity constraints \eqref{constr:SL-MPCC-CC1}-\eqref{constr:SL-MPCC-CC2}.
This observation suggests focusing on the relaxation} 
\begin{subequations}\label{eq:SL-MPCC-Relax}
\begin{align}
\optval{\kkt}{\rel}
= \max_{\substack{\bvaz,\bvax,\bvay,\bvaw\\ \bdvapi, \bdvamu,\bdvath}}\;\; &  \bdvapiotr \bvax - \bcoxtr \bvax - \bcoztr   \bvaz \\
\st\ & {(\bvaz,\bvax) \in \overline \feas,\; \eqref{constr:SL-MPCC-Leader2}-\eqref{constr:SL-MPCC-Stnry2}} 
\end{align}
\end{subequations}
obtained by removing complementarity constraints \eqref{constr:SL-MPCC-CC1}-\eqref{constr:SL-MPCC-CC2} and integrality requirements in \eqref{constr:SL-MPCC-Leader1}.
We keep the bilinear term from \eqref{constr:SL-MPCC-Obj} in relaxation \eqref{eq:SL-MPCC-Relax}. 
This is because we will be able to argue in Section~\ref{sec:KKTvsDUAL} that, even without relaxing this bilinear term, relaxation \eqref{eq:SL-MPCC-Relax} can be weaker than relaxation \eqref{eq:SL-LagrDual-Relax} presented in Section~\ref{sec:SingleLevel_Reform2}.

\begin{remark}
    Constraints \eqref{constr:SL-MPCC-Leader2}-\eqref{constr:SL-MPCC-Stnry2} are linear, except possibly for \eqref{constr:SL-MPCC-Stnry1}
    that contains function $\myPhi(\bvay)$, which can be nonlinear. 
    When $\myPhi(\bvay)$ is affine and {$\overline \feas$ is polyhedron}, then \eqref{eq:SL-MPCC-Relax} consists of optimizing a bilinear objective function over a polyhedral feasible region.
\end{remark}

%%%%%%%%%%%%%%%%%%%%%%%%%%%%%%%%%%
\subsection{{Duality}-based {formulation}}
\label{sec:SingleLevel_Reform2}
%%%%%%%%%%%%%%%%%%%%%%%%%%%%%%%%%%
{
In this section, we obtain a second reformulation in Theorem~\ref{theorem:SL-Reform2} using Lagrangian duality to rewrite the objective \eqref{constr:SL-MPCC-Obj} in \eqref{eq:SL-MPCC}. We start with the following ancillary result.
}
\begin{lemma}
{For a given leader's decision $\bvax$,} the dual of follower problem \eqref{eq:OPT-follower} is
\begin{subequations}\label{eq:LagrDual}
\begin{align}
        \max_{\bdvapi,\bdvamu,\bdvath} \;\; & \myphi_\text{d}^\bvax(\bdvapi,\bdvamu,\bdvath)
        \begin{aligned}[t]
        :=\ & \myphi_\text{p}\left(\myPhi^{-1}(\bdvamuy-\bdvathy-\lhsytr \bdvapi)\right) \\
        & - \left\langle \bdvamuy-\bdvathy-\lhsytr \bdvapi,\ \myPhi^{-1}(\bdvamuy-\bdvathy-\lhsytr \bdvapi) \right\rangle \\ 
        & -\ubvawtr\bdvathw - \ubvaytr\bdvathy - \brhsBtr \bdvapi - \bdvapiotr \bvax 
        \end{aligned}\label{constr:LagrDual-Obj} \\
        \st \;\; & \lhswtr \bdvapi + \bdvathw - \bdvamuw = 0,\ \bdvamu\geq0,\ \bdvath \geq 0. 
\end{align}
\end{subequations}
\end{lemma}

\begin{proof}
For a given leader's decision $\bvax$, the dual objective of follower problem is
\begin{align}\label{eq:DualObj-Opt}
 \myphi_\text{d}^\bvax(\bdvapi,\bdvamu,\bdvath)
    & = \min_{\bvay,\bvaw}\ \mathcal{L}_{\bvax}(\bvay,\bvaw;\bdvapi,\bdvamu,\bdvath) \nonumber \\
    & = - \ubvawtr \bdvathw - \ubvaytr \bdvathy - \brhsBtr \bdvapi - \bdvapiotr \bvax \nonumber \\
    & \quad + \min_{\bvaw}\ \mathcal{L}_{\bvax}^1(\bvaw;\bdvapi,\bdvamuw,\bdvathw) + \min_{\bvay}\ \mathcal{L}_{\bvax}^2(\bvay;\bdvapi,\bdvamuy,\bdvathy), 
\end{align}
where $\mathcal{L}_{\bvax}^1(\bvaw;\bdvapi,\bdvamuw,\bdvathw) = \left\langle \lhswtr \bdvapi + \bdvathw - \bdvamuw,\bvaw\right\rangle$ and  $\mathcal{L}_{\bvax}^2(\bvay;\bdvapi,\bdvamuy,\bdvathy)=\myphi_\text{p}(\bvay) + \left\langle \lhsytr \bdvapi + \bdvathy - \bdvamuy ,\bvay\right\rangle$.
As variables $\bvaw$ are unrestricted in sign in \eqref{eq:DualObj-Opt}, we observe that 
\[\min_{\bvaw}\ \mathcal{L}_{\bvax}^1(\bvaw;\bdvapi,\bdvamuw,\bdvathw) = \begin{cases}0 & \text{ if } \lhswtr \bdvapi + \bdvathw - \bdvamuw = 0\\ -\infty & \text{ otherwise.}\end{cases}
\]
Further, for fixed $\bdvamu\geq0$, $\bdvath\geq0$, and $\bdvapi$, consider the minimization problem 
$
\min_{\bvay}\ \mathcal{L}_{\bvax}^2(\bvay;\bdvapi,\bdvamuy,\bdvathy).
$
Let $\tilde{\bvay}^*$ be a minimizer. 
Assumption~\ref{assump:Unique} implies that $\mathcal{L}_{\bvax}^2(\cdot;\bdvapi,\bdvamuy,\bdvathy)$ is convex in $\bvay$ for given $\bdvapi,\bdvamu,\bdvath$. 
Hence, the following optimality condition is necessary and sufficient: {$\nabla_\bvay \mathcal{L}_{\bvax}^2(\tilde{\bvay}^*;\bdvapi,\bdvamuy,\bdvathy) = \myPhi(\tilde{\bvay}^*) + \lhsytr \bdvapi + \bdvathy - \bdvamuy = 0$}.
In fact, Assumption~\ref{assump:Unique} implies strict convexity of $\mathcal{L}^2_{\bvax}(\bvay;\bdvapi,\bdvamuy,\bdvathy)$ and invertibility of $\myPhi(\cdot)$. 
If $\left(\bdvapi,\bdvamu,\bdvath\right)$ is such that $\bdvamuy - \bdvathy - \lhsytr \bdvapi\in\mathrm{dom}\left(\myPhi^{-1}\right)$, the minimizer $\tilde{\bvay}^*$ exists and is uniquely given by $\tilde{\bvay}^* =  \myPhi^{-1} (\bdvamuy - \bdvathy - \lhsytr \bdvapi )$. 
According to Remark~\ref{remark:RefinedSlaterCondition}, the dual optimal value is attained
which means there exists $\left(\bdvapi,\bdvamu,\bdvath\right)$ such that $\bdvamuy - \bdvathy - \lhsytr \bdvapi\in\mathrm{dom}\left(\myPhi^{-1}\right)$. 
As a result, substituting $\tilde{\bvay}^* =  \myPhi^{-1} (\bdvamuy - \bdvathy - \lhsytr \bdvapi )$ yields the well-defined dual problem \eqref{eq:LagrDual}.
\end{proof}

\begin{theorem}\label{theorem:SL-Reform2}
% Model 
{The following model is a reformulation of \eqref{eq:SL-MPCC},}
\begin{subequations} \label{eq:SL-LagrDual}
\begin{align}
\optval{\dual}{}
= \max_{\substack{\bvaz,\bvax,\bvay,\bvaw\\ \bdvapi, \bdvamu,\bdvath}}\;\;
& - \left\langle \myPhi(\bvay),\ \bvay \right\rangle - \ubvawtr \bdvathw - \ubvaytr \bdvathy - \brhsBtr \bdvapi -\bcoxtr \bvax - \bcoztr   \bvaz \label{eq:SL-LagrDual-Obj}\\
\st\;\; & \eqref{constr:SL-MPCC-Leader1} - \eqref{constr:SL-MPCC-CC2}
\end{align}
\end{subequations}
% is a reformulation of \eqref{eq:SL-MPCC},
{\textit{i.e.}, it has {the} same optimal value and optimal solutions as \eqref{eq:SL-MPCC}}.
\end{theorem}
\begin{proof}
Recall that for any given leader decision $\bvax$, the follower decisions $(\bvay;\bdvapi,\bdvamu,\bdvath)$ must be primal-dual optimal. 
From Remark \ref{remark:RefinedSlaterCondition}, strong duality holds and therefore equality 
$\myphi_\text{p}(\bvay) = \myphi_\text{d}^{\bvax}(\bdvapi,\bdvamu,\bdvath)$
must also hold.
Substituting the expression for $\myphi_\text{d}^{\bvax}(\bdvapi,\bdvamu,\bdvath)$ from \eqref{constr:LagrDual-Obj} and rearranging gives
\begin{align}\label{eq:LL-Duality}
& \begin{aligned}[t] \bdvapiotr \bvax = \myphi_\text{p}\left(\myPhi^{-1}(\bdvamuy-\bdvathy-\lhsytr \bdvapi)\right) - \left\langle \bdvamuy-\bdvathy-\lhsytr \bdvapi,\ \myPhi^{-1}(\bdvamuy-\bdvathy-\lhsytr \bdvapi) \right\rangle \\ 
- \ubvawtr\bdvathw - \ubvaytr\bdvathy - \brhsBtr \bdvapi - \myphi_\text{p}(\bvay). 
\end{aligned}
\end{align}
Using \eqref{eq:LL-Duality} to substitute the bilinear term $\bdvapiotr \bvax$ in \eqref{constr:SL-MPCC-Obj}, we reformulate \eqref{eq:SL-MPCC} as
\begin{subequations} \label{eq:SL-MPCC2}
\begin{align}
\max_{\substack{\bvaz,\bvax,\bvay,{\bvaw},\\ \bdvapi, \bdvamu,\bdvath}}\;\; 
& \begin{aligned}[t]
\myphi_\text{p}\left(\myPhi^{-1}(\bdvamuy-\bdvathy-\lhsytr \bdvapi)\right) - \left\langle \bdvamuy-\bdvathy-\lhsytr \bdvapi,\ \myPhi^{-1}(\bdvamuy-\bdvathy-\lhsytr \bdvapi) \right\rangle \\
- \myphi_\text{p}(\bvay) - \ubvawtr \bdvathw - \ubvaytr \bdvathy - \brhsBtr \bdvapi -\bcoxtr \bvax - \bcoztr   \bvaz \end{aligned}\label{constr:SL-MPCC2-Obj}\\
\st\ & \eqref{constr:SL-MPCC-Leader1}-\eqref{constr:SL-MPCC-CC2}.
\end{align}
\end{subequations}
Using \eqref{constr:SL-MPCC-Stnry1} to substitute $\bdvamuy-\bdvathy-\lhsytr \bdvapi$ for $\myPhi(\bvay)$ in \eqref{constr:SL-MPCC2-Obj} gives \eqref{eq:SL-LagrDual-Obj}. 
\end{proof}

\begin{remark}
Formulations \eqref{eq:SL-MPCC} and \eqref{eq:SL-LagrDual} differ only in their objective functions. 
\end{remark}

{Similar to the relaxation of the KKT-based formulation we introduced in Section~\ref{sec:SingleLevel_Reform1}, we now consider the} relaxation of \eqref{eq:SL-LagrDual} obtained after relaxing its complementarity constraints \eqref{constr:SL-MPCC-CC1}-\eqref{constr:SL-MPCC-CC2} and integrality constraints \eqref{constr:SL-MPCC-Leader1}:
\begin{subequations}\label{eq:SL-LagrDual-Relax}
\begin{align}
\optval{\dual}{\rel}
= \max_{\substack{\bvaz,\bvax,\bvay,\bvaw\\ \bdvapi, \bdvamu,\bdvath}}\;\; &  - \left\langle \myPhi(\bvay),\ \bvay \right\rangle - \ubvawtr \bdvathw - \ubvaytr \bdvathy - \brhsBtr \bdvapi -\bcoxtr \bvax - \bcoztr   \bvaz 
\label{constr:SL-LagrDual-Relax-Obj}\\
\st\ & {(\bvaz,\bvax) \in \overline \feas,\; \eqref{constr:SL-MPCC-Leader2}-\eqref{constr:SL-MPCC-Stnry2}}. 
\end{align}
\end{subequations}

%%%%%%%%%%%%%%%%%%%%%%%%%%%%%%%%%%
\subsection{On the strength of KKT-based and Duality-based relaxations}
\label{sec:KKTvsDUAL}
%%%%%%%%%%%%%%%%%%%%%%%%%%%%%%%%%%

{We first establish that the relaxation bound from~\eqref{eq:SL-MPCC-Relax} cannot be smaller than the relaxation bound obtained from~\eqref{eq:SL-LagrDual-Relax}.}

\begin{lemma}
It holds that
$\optval{\dual}{\rel} \le \optval{\kkt}{\rel}$.
\end{lemma}
\begin{proof}
    {C}onstraints \eqref{constr:SL-MPCC-PF1}-\eqref{constr:SL-MPCC-Stnry2} ensure primal-dual feasibility of the follower problem. 
    Weak duality then implies that $\myphi_\text{p}(\bvay) \geq \myphi_\text{d}^{\bvax}(\bdvapi,\bdvamu,\bdvath)$ holds for both \eqref{eq:SL-MPCC-Relax} and \eqref{eq:SL-LagrDual-Relax}.
    Thus,
%It follows that
\begin{align}\label{eq:LL-WeakDuality}
     & 
    \begin{array}{lll} 
    \bdvapiotr \bvax &\overset{\eqref{constr:LagrDual-Obj}}{\geq}& \myphi_\text{p}\left(\myPhi^{-1}(\bdvamuy-\bdvathy-\lhsytr \bdvapi)\right) \\ 
    &&- \left\langle \bdvamuy-\bdvathy-\lhsytr \bdvapi,\ \myPhi^{-1}(\bdvamuy-\bdvathy-\lhsytr \bdvapi) \right\rangle \\ 
    &&- \ubvawtr\bdvathw - \ubvaytr\bdvathy - \brhsBtr \bdvapi - \myphi_\text{p}(\bvay).
    \end{array}
    \end{align}
Using constraint \eqref{constr:SL-MPCC-Stnry1} to substitute $\bdvamuy-\bdvathy-\lhsytr \bdvapi$ for $\myPhi(\bvay)$ in \eqref{eq:LL-WeakDuality} yields 
\begin{align*}
& \bdvapiotr \bvax\ {\geq} - \left\langle \myPhi(\bvay),\ \bvay \right\rangle - \ubvawtr \bdvathw - \ubvaytr \bdvathy - \brhsBtr \bdvapi, 
\end{align*}
which shows that 
$\optval{\kkt}{\rel} \ge \optval{\dual}{\rel}$.
\end{proof}

We argue next that the difference between 
$\optval{\dual}{\rel}$ and $\optval{\kkt}{\rel}$ can be significant, even for the case where the cost vector $\myPhi$ is affine. 

%%%%%%%%%%%%%%%%%
\subsubsection{Strength of relaxations when  $\Phi$ is affine}
\label{sec:AffineCase}
%%%%%%%%%%%%%%%%%

Consider $\myPhi(\bvay) = {\mathcal{R}}\bvay+{r}$ where
{${\mathcal{R}}$ is a positive definite matrix} so that Assumptions \ref{assump:JacobSymm} and \ref{assump:Unique} are satisfied. 
The follower objective is $\myphi_{p}(\bvay)=\frac{1}{2}\bvay\tr {\mathcal{R}}\bvay+{r}\tr \bvay$ and \eqref{constr:SL-MPCC-Stnry1} is affine. 

{
We next introduce some notation. Let $I^\bvay$ and $I^\bvaw$ represent the identity matrices of dimensions $\dim{\bvay}$ and $\dim{\bvaw}$, respectively. 
% Define ${Y}_\setUnbd = \{i \in [\dim{\bvay}]:\ubvay_i = \infty\}$.
% For matrices $M$ with $\dim{\bvay}$ columns (e.g. $\lhsy$, ${\mathcal{R}}$, $I^\bvay$), we partition $M$ into the submatrix $M_\setUnbd$ whose column indices {are such that $\{i \in [\dim{\bvay}]:\ubvay_i = \infty\}$}  %belong to %${Y}_\setUnbd$ 
% and submatrix $M_\setBd$ whose column indices do not. 
{
For a matrix $M$ with $\dim{\bvay}$ columns (e.g. $\lhsy$, ${\mathcal{R}}$, $I^\bvay$), we define $M_\setUnbd$ (resp. $M_\setBd$) to be the submatrix obtained by selecting the columns of $M$ whose indices belong (resp. do not belong) to the set $\{i \in [\dim{\bvay}]:\ubvay_i = \infty\}$.
Likewise, for a matrix $N$ with $\dim{\bvaw}$ columns (e.g. $\lhsw$, $I^\bvaw$), we define submatrices $N_\setUnbd$ and $N_\setBd$ based on $\{i \in [\dim{\bvaw}]:\ubvaw_i = \infty\}$.}
% we 
% % define ${W}_\setUnbd = \{i \in [\dim{\bvaw}]:\ubvaw_i = \infty\}$, and 
% partition a matrix $N$ with $\dim{\bvaw}$ columns (e.g. $\lhsw$, $I^\bvaw$) into submatrices $N_\setUnbd$ and $N_\setBd$ based on {$\{i \in [\dim{\bvaw}]:\ubvaw_i = \infty\}$}. 
We partition vectors $v\in\Re^{\dim{\bvay}}$ or $\Re^{\dim{\bvaw}}$ similarly into subvectors $v_\setUnbd$ and $v_\setBd$. 
Complementarity constraints \eqref{constr:SL-MPCC-CC1}-\eqref{constr:SL-MPCC-CC2} imply that $\bdvath_\setUnbd^\bvay = 0$ and $\bdvath_\setUnbd^\bvaw = 0$.
Hence, the constraint set \eqref{constr:SL-MPCC-PF1}-\eqref{constr:SL-MPCC-Stnry2} in relaxations \eqref{eq:SL-MPCC-Relax} and \eqref{eq:SL-LagrDual-Relax} becomes
\begin{subequations}\label{constrs:AffineCase}
    \begin{align}
    \label{constr:SLv2-MPCC-PF1} &
    {\begin{pmatrix}
     G_\setUnbd & H_\setUnbd & 0 & 0 & 0 & 0 & 0\\
    {\mathcal{R}}_\setUnbd & 0 & \lhsytr & -I^\bvay & 0 & I_\setBd^\bvay & 0\\
    0 & 0 & \lhswtr & 0 & -I^\bvaw & 0 & I_\setBd^\bvaw
    \end{pmatrix}
    \begin{pmatrix}
    {\bvay}_\setUnbd \\ {\bvaw}_\setUnbd \\ \bdvapi \\ {\bdvamu}^\bvay \\ {\bdvamu}^\bvaw \\ 
    {\bdvath}_\setBd^\bvay \\ {\bdvath}_\setBd^\bvaw
    \end{pmatrix} +
     \begin{pmatrix}
     G_\setBd & H_\setBd \\
    {\mathcal{R}}_\setBd & 0 \\
    0 & 0
    \end{pmatrix}
    \begin{pmatrix}
    {\bvay}_\setBd \\ {\bvaw}_\setBd
    \end{pmatrix}
    % = \begin{pmatrix}
    %     \rhsBo + \bvax \\ \rhsBi \\ -{r}
    % \end{pmatrix}
    = {\begin{pmatrix}
        \rhsB + \begin{pmatrix} \bvax \\ 0 \end{pmatrix} \\
        -{r} \\ 0
    \end{pmatrix}}
    } \\
    \label{constr:SLv2-MPCC-PF2} & \bvay \geq 0,\;\; \bvaw \geq 0,\;\; \bvay_\setBd \leq \ubvay_\setBd,\;\; \bvaw_\setBd \leq \ubvaw_\setBd  \\
    \label{constr:SLv2-MPCC-DF} & \bdvamu\geq0,\;\; {{\bdvath}_\setBd^\bvay\geq 0,\;\; {\bdvath}_\setBd^\bvaw \geq 0}{.}
    \end{align}
\end{subequations}
}
Relaxations \eqref{eq:SL-MPCC-Relax} and \eqref{eq:SL-LagrDual-Relax} can thus be written as 
\begin{subequations}\label{eq:SLv2-MPCC-Relax}
\begin{align}
\optval{\kkt}{\rel}
= {\max_{\substack{\bvaz,\bvax,\bvay,\bvaw\\ \bdvapi, \bdvamu, {\bdvath}_\setBd^\bvay, {\bdvath}_\setBd^\bvaw}}}\;\; &  \bdvapiotr \bvax - \bcoxtr \bvax - \bcoztr   \bvaz \\
\st\qquad & \ {(\bvaz,\bvax) \in \overline \feas,\; \eqref{constr:SL-MPCC-Leader2}-\eqref{constr:SL-MPCC-Leader3},\; \eqref{constr:SLv2-MPCC-PF1}-\eqref{constr:SLv2-MPCC-DF}
}
\end{align}
\end{subequations}
and
\begin{subequations}\label{eq:SLv2-LagrDual-Relax}
\begin{align}
\optval{\dual}{\rel}
= {\max_{\substack{\bvaz,\bvax,\bvay,\bvaw\\ \bdvapi, \bdvamu, {\bdvath}_\setBd^\bvay, {\bdvath}_\setBd^\bvaw}}}\;\; &  - \bvay\tr {\mathcal{R}}\bvay - {r}\tr \bvay 
{- \overline{\bvaw}_\setBd\tr {\bdvath}_\setBd^\bvaw - \overline{\bvay}_\setBd\tr {\bdvath}_\setBd^\bvay}
- \brhsBtr \bdvapi -\bcoxtr \bvax - \bcoztr   \bvaz 
\label{constr:SLv2-LagrDual-Relax-Obj}\\
\st\qquad & \ {(\bvaz,\bvax) \in \overline \feas,\; \eqref{constr:SL-MPCC-Leader2}-\eqref{constr:SL-MPCC-Leader3},\; \eqref{constr:SLv2-MPCC-PF1}-\eqref{constr:SLv2-MPCC-DF}}.
\label{constr:SLv2-LagrDual-Relax-Constr}
\end{align}
\end{subequations} 

\begin{lemma}\label{lemma:Unbd_ReformI}
{ Let $(\dot{\bvaz}, \dot{\bvax})\in \overline \feas$ be such that \eqref{constr:SL-MPCC-Leader2} is satisfied. 
 Let $(\mymark{\bvay}, \mymark{\bvaw}, \markbdvapi, \mymark{\bdvamu}^\bvaw, \mymark{\bdvamu}^\bvaw,\mymark{\bdvath}_\setBd^\bvay,\mymark{\bdvath}_\setBd^\bvaw)$ be a nonzero nonnegative vector satisfying $\mymark{\bvay}_\setBd=0$, $\mymark{\bvaw}_\setBd=0$, and
 \begin{align}\label{eq:SystemOfEqns_UnbdRay}
    \begin{pmatrix}
     G_\setUnbd & H_\setUnbd & 0 & 0 & 0 & 0 & 0\\
    {\mathcal{R}}_\setUnbd & 0 & \lhsytr & -I^\bvay & 0 & I_\setBd^\bvay & 0\\
    0 & 0 & \lhswtr & 0 & -I^\bvaw & 0 & I_\setBd^\bvaw
    \end{pmatrix}
    \begin{pmatrix}
        \mymark{\bvay}_\setUnbd \\ \mymark{\bvaw}_\setUnbd \\ \markbdvapi \\ \mymark{\bdvamu}^\bvay \\ \mymark{\bdvamu}^\bvaw \\ 
        \mymark{\bdvath}_\setBd^\bvay \\ \mymark{\bdvath}_\setBd^\bvaw
    \end{pmatrix} = 0,
\end{align}
\ie this vector is a \textbf{ray} corresponding to constraint set 
{
\eqref{constr:SL-MPCC-Leader3}, \eqref{constr:SLv2-MPCC-PF1}-\eqref{constr:SLv2-MPCC-DF}.
}
If $\markbdvapiotr \dot\bvax>0$, 
then $\optval{\kkt}{\rel} =\infty$, \ie relaxation \eqref{eq:SLv2-MPCC-Relax} is unbounded.}
\end{lemma}
\begin{proof}
For the given {$(\dot{\bvaz},\dot{\bvax})$}, {Assumptions~\ref{assump:Unique} and \ref{assump:FollowerFeas} imply} that there exists $(\dot{\bvay},\dot{\bvaw})$  and corresponding dual variables $(\dot{ \bdvapi},\dot{\bdvamu},\dot{\bdvath})$ satisfying constraints \eqref{constr:SLv2-MPCC-PF1}-\eqref{constr:SLv2-MPCC-DF}.
For $\rho>0$, define the solutions
$(\bvax^{(\rho)}, \bvaz^{(\rho)}) = (\dot{\bvax}, \dot{\bvaz})$,
$\bvay^{(\rho)} = \dot{\bvay} + \rho \cdot \mymark{\bvay}$, 
$\bvaw^{(\rho)}=\dot{\bvaw} + \rho \cdot \mymark{\bvaw}$, 
$\bdvapi^{(\rho)}=\dot{\bdvapi}+\rho \cdot \mymark{\bdvapi}$, 
${\bdvamu^{\bvay}}^{(\rho)}={\dot{\bdvamu}^{\bvay}}+\rho \cdot \mymark{\bdvamu}^\bvay $,  
${\bdvamu^{\bvaw}}^{(\rho)}={\dot{\bdvamu}^{\bvaw}}+\rho  \cdot \mymark{\bdvamu}^\bvaw$,
{${{\bdvath}_\setBd^\bvay}^{(\rho)} = {\dot {\bdvath}_\setBd^\bvay} + \rho \cdot \mymark{\bdvath}_\setBd^\bvay$, and
${{\bdvath}_\setBd^\bvaw}^{(\rho)} = {\dot {\bdvath}_\setBd^\bvaw} + \rho \cdot \mymark{\bdvath}_\setBd^\bvaw$}.
It is easy to verify that these solutions satisfy 
\eqref{constr:SL-MPCC-Leader3}, \eqref{constr:SLv2-MPCC-PF1}-\eqref{constr:SLv2-MPCC-DF}.
Denote the objective function of the solution associated with $\rho$ by $\optval{\kkt}{\rho}$. 
Then, it can be verified that $\optval{\kkt}{\rho} = \optval{\kkt}{0} + \rho \markbdvapiotr \dot\bvax$.
Since $\markbdvapiotr \dot\bvax>0$, the optimal value grows without bound as $\rho \to \infty$, \ie $\optval{\kkt}{\rel}=\infty$.
\end{proof}

\begin{remark}
When its conditions are satisfied, Lemma~\ref{lemma:Unbd_ReformI} suggests that branch-and-bound will likely struggle in solving Formulation~\eqref{eq:SL-MPCC}, as the problem relaxation at the root node will be unbounded (barring success from generic cuts or pre-processing routines at otherwise bounding the objective.)
This has the potential to significantly slow down further search as branching decisions will be harder to make and many nodes will need to be explored before a reasonable upper bound is obtained.
\end{remark}

\begin{lemma}\label{lemma:BndRelaxII}
Assume $\brhsB=\bm{0}$. Then $\optval{\dual}{\rel}<\infty$, \ie relaxation \eqref{eq:SLv2-LagrDual-Relax} has a bounded optimal value.
\end{lemma}
\begin{proof}
    The following terms in objective function \eqref{constr:SLv2-LagrDual-Relax-Obj} are bounded over the constraints 
   \eqref{constr:SLv2-LagrDual-Relax-Constr}: 
    ${- \overline{\bvaw}_\setBd\tr {\bdvath}_\setBd^\bvaw \leq 0}$, 
    ${- \overline{\bvay}_\setBd\tr {\bdvath}_\setBd^\bvay \leq 0}$,
    $-\bcoxtr \bvax\leq-\min_{0\leq\bvax\leq \ubvax} \bcoxtr \bvax$, and
    $-\bcoztr   \bvaz \leq -\min_{{0\leq\bvaz\leq \ubvaz}} \bcoztr   \bvaz$.
    For $\brhsB=\bm{0}$, we have that $\brhsBtr \bdvapi=0$. 
    Define $\widetilde \myphi_\text{p}(\bvay):=\bvay\tr {\mathcal{R}}\bvay+{r}\tr \bvay$.
    Since {${\mathcal{R}}$ is positive definite} then $\widetilde \myphi_\text{p}(\bvay)$ is coercive which implies that it has a global minimizer on $\Re^{\dim{\bvay}}$, \ie
    $\widetilde \myphi_\text{p}(\bvay)\geq \min_{\bvay'} \widetilde \myphi_\text{p}(\bvay') > -\infty$. 
    Thus, objective function \eqref{constr:SLv2-LagrDual-Relax-Obj} is bounded above over \eqref{constr:SLv2-LagrDual-Relax-Constr} as
    \begin{align*}
        & - \widetilde \myphi_\text{p}(\bvay) {- \overline{\bvaw}_\setBd\tr {\bdvath}_\setBd^\bvaw - \overline{\bvay}_\setBd\tr {\bdvath}_\setBd^\bvay} - \brhsBtr \bdvapi -\bcoxtr \bvax - \bcoztr   \bvaz \\ 
        & \leq -\textstyle\min_{\bvay'\in\Re^{\dim{\bvay}}} \widetilde \myphi_\text{p}(\bvay')-\textstyle \min_{0\leq\bvax\leq \ubvax} \bcoxtr \bvax -\textstyle \min_{{0\leq\bvaz\leq \ubvaz}} \bcoztr   \bvaz \;\;
        < \infty.
    \end{align*}
\end{proof}

\begin{remark}
    When $\brhsB=\bm{0}$, Lemma~\ref{lemma:BndRelaxII} establishes that even if there exists 
    {$(\mymark{\bvay}, \mymark{\bvaw}, \markbdvapi, \mymark{\bdvamu}^\bvaw, \mymark{\bdvamu}^\bvaw, \mymark{\bdvath}_\setBd^\bvay, \mymark{\bdvath}_\setBd^\bvaw)\geq0$} satisfying \eqref{eq:SystemOfEqns_UnbdRay} and such that $\markbdvapiotr \dot\bvax>0$, 
    then relaxation \eqref{eq:SLv2-LagrDual-Relax} is bounded, which is a significant advantage over relaxation \eqref{eq:SLv2-MPCC-Relax}.
\end{remark}

\begin{remark}
    Lemma~\ref{lemma:BndRelaxII} generalizes to  
    non-affine vector functions $\myPhi(\bvay)$ (satisfying Assumptions \ref{assump:JacobSymm}-\ref{assump:Unique}) %-\ref{assump:OpenDomain}) 
    for which $\widetilde \myphi_\text{p}(\bvay):=\left\langle \myPhi(\bvay),\ \bvay \right\rangle$ is coercive. 
    An example is 
    $\myPhi(\bvay)=\big({1}/{\sqrt{1+\bvay_i^2}} + 2\bvay_i,\ {i\in \dim{\bvay}}\big)$ for which $\widetilde\myphi_\text{p}(\bvay) =\sum_{i=1}^{\dim{\bvay}}({\bvay_i}/{\sqrt{1+\bvay_i^2}} + 2\bvay_i^2)$ is coercive.    
\end{remark}

%%%%%%%%%%%%%%%%%
\subsubsection{Example}
\label{sec:ToyExample}
%%%%%%%%%%%%%%%%%

We next illustrate the difference in the strength of the two formulations presented above on a simple instance of the {equilibrium facility location (EFL)} problem that we discuss in detail 
in Section~\ref{sec:EFL-Descrp}.
The instance we consider has variables 
$\bvaz \in \Re$, $\bvax \in \Re$, 
$\bvay{=\begin{pmatrix}\bvay_a & \bvay_b & \bvay_c & \bvay_d\end{pmatrix}\tr}\in\Re^4$ and has no variable  $\bvaw$. {
Further, we choose $\overline \feas=\left\{(\bvaz,\bvax) \in \Re^2 \ \big| \ \bvax -10\bvaz\leq 0\right\}$, $\mathcal{Z}=\mathbb{Z}$, and $\mathcal{X}=\Re$.
% The sets $\mathcal{Z}:=\mathbb{Z}$, $\mathcal{X}:=\Re$, and $\overline \feas=\left\{(\bvaz,\bvax) \in \Re^2 \ \big| \ \bvax -10\bvaz\leq 0\right\}$
}
Bounds on the variables are %chosen so that 
$\ubvax=7.5$, {$\ubvaz=1$}, and $\ubvay=\infty$.
The constraints of the follower set are defined by
$\lhsyo = \begin{pmatrix} 1 & 0 & -1 & 0 \end{pmatrix}$, $\rhsBo=0$, $\lhsyi = \begin{pmatrix} -1 & 1 & 0 & -1 \end{pmatrix}$, and $\rhsBi=0$,
 whereas $\lhswo$ and $\lhswi$ are not defined since there are no variable $\bvaw$. 
The objective functions are defined through $\myPhi(\bvay)=\begin{pmatrix}\bvay_a + 10 & \bvay_b - 20 & \bvay_c + 10 & \bvay_d + 20\end{pmatrix}\tr$ and the cost parameters are %chosen so that 
$\bcox=0.5$, and $\bcoz=0.5$. 
The constraints of the KKT-based and Duality-based formulations are
\begin{subequations}
\begin{align}    
    & \bvax -10\bvaz\leq 0, \;\; \bvaz\in\{0,1\},\;\; 0 \leq \bvax \leq 7.5, \nonumber \\
    & {\bdvapi \geq 0} \label{constr:Example-Leader1} \\
    & \bvay_a - \bvay_c - \bvax = 0, \;\; -\bvay_a + \bvay_b - \bvay_d = 0, \nonumber\\
    & \bvay_a \geq 0, \;\; \bvay_b \geq 0,\;\; \bvay_c \geq 0, \;\; \bvay_d \geq 0 \\
    & \bdvamu_{a}\geq 0, \;\; \bdvamu_{b}\geq 0, \;\; \bdvamu_{c}\geq 0, \;\; \bdvamu_{d}\geq 0 \\
    &\bvay_a + 10 + \bdvapi_0 - \bdvapi_1 - \bdvamu_{a} = 0, \nonumber\\
    & \bvay_b - 20 + \bdvapi_1 - \bdvamu_{b} = 0, \nonumber \\
    & \bvay_c + 10 - \bdvapi_0 - \bdvamu_{c} = 0, \nonumber \\
    & \bvay_d + 20 - \bdvapi_1 - \bdvamu_{d} = 0 \\ 
    & \{\bvay_a,\bdvamu_a\}\; \text{ is\; SOS1},\;\; \{\bvay_b,\bdvamu_b\}\; \text{ is\; SOS1}, \nonumber\\
    & \{\bvay_c,\bdvamu_c\}\; \text{ is\; SOS1},\;\; \{\bvay_d,\bdvamu_d\}\; \text{ is\; SOS1}. \label{constr:Example-CC}
\end{align}
\end{subequations}
The KKT-based formulation has objective $\optval{\kkt}{} = \max \;\;(\bdvapi_0 - 0.5) \bvax - 0.5 \bvaz$ whereas the Duality-based formulation has objective $\optval{\dual}{}=
    \max\;\;  - \bvay_{a}^2 - \bvay_b^2 - \bvay_c^2 - \bvay_d^2 - 10 \bvay_a + 20 \bvay_b  - 10 \bvay_c - 20 \bvay_d - 0.5 \bvax - 0.5 \bvaz$.
In this example, \eqref{eq:SystemOfEqns_UnbdRay} reduces to {$\begin{pmatrix}
    \lhsy & 0 & 0\\
    {\mathcal{R}} & \lhsytr & -I \\
\end{pmatrix}
\begin{pmatrix}
        \mymark{\bvay} \\ \markbdvapi \\ \mymark{\bdvamu}
    \end{pmatrix} = 0$}
where $\lhsy =\begin{pmatrix} \ 1 & \ 0 & -1 & \ 0 \\ -1 & \ 1 & \ 0 & -1 \end{pmatrix}$ and ${\mathcal{R}} = I$.    
The vectors $\mymark{\bvay}=\begin{pmatrix} 1 & 1 & 1 & 0 \end{pmatrix}^\top$, 
$\widetilde\bdvapi=\begin{pmatrix} 1 & 0 \end{pmatrix}^\top$, and $\mymark{\bdvamu}=\begin{pmatrix} 2 & 1 & 0 & 0 \end{pmatrix}^\top$ satisfy the sufficient condition for unboundedness of relaxation \eqref{eq:SL-MPCC-Relax} given in Lemma~\ref{lemma:Unbd_ReformI}. 
Further, since $\widetilde \myphi_\text{p}(\bvay)=\bvay_{a}^2 + \bvay_b^2 + \bvay_c^2 + \bvay_d^2 + 10 \bvay_a -20 \bvay_b + 10 \bvay_c + 20 \bvay_d$ is coercive and $\brhsB=\bm{0}$, the sufficient conditions for boundedness of relaxation \eqref{eq:SL-LagrDual-Relax} given in Lemma~\ref{lemma:BndRelaxII} are also satisfied. 

{Solving this example by branch-and-bound with GUROBI (v9.5.2) shows that it has an optimal value of $10.78125$. 
The Duality-based formulation has a root relaxation bound of 11.12347 whereas the root relaxation of {the} KKT-based formulation is unbounded. 
Further, 31 branch-and-bound nodes are required to solve the KKT-based formulation  
whereas the Duality-based formulation is solved at the root node.}

%%%%%%%%%%%%%%%%%%%%%%%%%%%%%%%%%%%%%%%%%%%%%%%%%%%%%%%%%%%%%%%%%%%%
\section{Numerical experiments}
\label{sec:NumericExp}
%%%%%%%%%%%%%%%%%%%%%%%%%%%%%%%%%%%%%%%%%%%%%%%%%%%%%%%%%%%%%%%%%%%%

In this section, we study the performance of the single-level reformulations described in Section~\ref{sec:SingleLevel} on two applications.
In Section~\ref{sec:EFL-Descrp}, we consider an equilibrium facility location (EFL) problem on networks.
In Section~\ref{sec:RGU-Descrp}, we consider a location problem for {renewable generation units planning} (RGUP) in distribution networks under uncertainty. 
The computational details and test instances used for the two applications are described in Sections~\ref{sec:EFL-InstGen} and \ref{sec:RGU-InstGen}, respectively. 
The results and insights gained from the computational experiments are discussed in Sections~\ref{sec:EFL-ComputResults1}, \ref{sec:EFL-ComputResults2}, and~\ref{sec:RGU-ComputResults1}.
{For EFL, we test a total of 60 medium-sized and 20 large-sized instances. 
Moreover, for {RGUP}, we test a total of 20 large-sized difficult IEEE instances from the power systems literature.}

%%%%%%%%%%%%%%%%%%%%%%%%%%%%%%%%%%
\subsection{Application 1: {EFL} on networks} %Equilibrium facility location
\label{sec:EFL-Descrp}
%%%%%%%%%%%%%%%%%%%%%%%%%%%%%%%%%%

Consider a directed network $G=(\nodes,\arcs)$ \big(where $\nodes$ is the set of nodes and $\arcs$ is the set of arcs\big) with existing demand and supply nodes for a single commodity denoted as $\nodesD \subseteq \nodes$ and $\nodesS\subseteq \nodes$, respectively. 
A leader firm wishes to locate production facilities at a subset of a set of potential nodes, say $\nodeso \subseteq \nodes$, of the network and determine their production levels subject to capacity constraints with the goal of maximizing profit. 
Let $\bvaz=(\bvaz_i,\ i\in \nodeso)$ denote the vector of binary decisions of locating a production facility at $i$ {having opening cost of $\conewz_i$} and let $\bvaq=(\bvaq_i,\ i\in\nodeso)$ denote the vector of production quantities at facility $i$ with unit production cost of $\coq_i$ and production capacity of $\ubvaq_i$ for $i\in\nodeso$.
The total capacity budget for production facilities is $\vaqmax$.
{Hence, the upper-level feasible set is described by $\feas_\textrm{EFL} = 
\left\{
 (\bvaz,\bvaq) \in \Re^{2|\nodeso|} \ \big|\  
 0 \leq \bvaq_i \leq \ubvaq_i  \bvaz_i \; \forall i\in \nodeso, \;
 \bvaz \in \{0,1\}^{|\nodeso|},\;
  \mathbf{1}\tr \bvaq \leq \vaqmax
 \right\}
$.}

At the lower level, we let $\bvaflo=(\bvaflo_{ij},\ (i,j)\in\arcs)$, $\bvadem=(\bvadem_i,\ i\in\nodesD)$, $\bvasup=(\bvasup_j,\ j\in\nodesS)$ be the vectors of flow, demand, and supply in the network, respectively. 
Assuming competition between new firms and those already in place, the production decisions $\bvaq$ change the total supply of commodity in the market and impact equilibrium prices and flow in the network. 
The resulting commodity flows and their corresponding demand and supply quantities correspond to a new competitive equilibrium {denoted by $(\bvaflo^*,\bvadem^*,\bvasup^*)$} which is obtained by solving the variational inequality 
\begin{align*}
    \langle \mybalpha(\bvaflo^*), \bvaflo-\bvaflo^*\rangle - \langle \mybbeta(\bvadem^*),\bvadem-\bvadem^*\rangle + \langle \mybgamma(\bvasup^*), \bvasup-\bvasup^*\rangle \geq 0, \quad \forall (\bvaflo, \bvadem, \bvasup) \in \Omega(\bvaq),
\end{align*}
where $\mybalpha(\bvaflo)$, $\mybbeta(\bvadem)$, $\mybgamma(\bvasup)$ are the inverse flow, inverse demand, and inverse supply cost vector functions, respectively; see~\cite{tobin1986spatial}. 
The set $\myOmega(\bvaq)$ is comprised of the network flow balance and non-negativity constraints
\begin{eqnarray*}
    \myOmega(\bvaq) = 
    \left\{
    \begin{array}{l|l}
    (\bvaflo,\bvadem,\bvasup)\; & 
    \begin{array}{ll} 
    & \idyof{} \bvaflo + \idyoD{} \bvadem - \idyoS{} \bvasup - \bvaq = 0, \\
    & \idyif{} \bvaflo + \idyiD{} \bvadem - \idyiS{} \bvasup = 0,\\
    & \bvaflo \geq 0,\; \bvadem \geq 0,\; \bvasup \geq 0 
    \end{array}
    \end{array} 
    \right\},
\end{eqnarray*}
where $(\idyof{})_{\nodeso\times \arcs}$, $(\idyoD{})_{\nodeso\times \nodesD}$, $(\idyoS{})_{\nodeso\times \nodesS}$ are the node-arc incidence, demand node incidence, and supply node incidence matrices corresponding to the nodes in $\nodeso$, respectively, and $\idyif{}$, $\idyiD{}$, $\idyiS{}$ are similar matrices corresponding to the nodes in $\nodesi=\nodes\setminus\nodeso$.
{The node-arc incidence matrix (here, $\idyof{}$ and $\idyif{}$) is commonly used in literature \cite[pgs.3-6]{van2010graph} to represent network flow constraints. 
For $i\in\nodeso$ and $j\in\nodesD$, the demand node incidence matrix  $\idyoD{}$ has entry $[\idyoD{}]_{ij}$ equal to $1$ if $i=j$ and equal to $0$ otherwise.
Other incidence matrices $\idyoS{},\idyiD{}$, $\idyiS{}$ are defined similarly.} {
The KKT-based and the Duality-based formulations of EFL can be obtained from \eqref{eq:SL-MPCC} and \eqref{eq:SL-LagrDual} using the %application 
specific constraints described above. Their explicit forms 
%can be found 
is in Appendix~\ref{sec:EFL-Appendix_SingleLevel}. 
}

%%%%%%%%%%%%%%%%%
\subsubsection{Instance generation}
\label{sec:EFL-InstGen}
%%%%%%%%%%%%%%%%%

We construct networks with different number $|\nodes|$ of nodes and number $|\arcs|$ of arcs. 
The arcs are randomly generated and the sets $\nodeso$, $\nodesD$, $\nodesS$ are randomly chosen with sizes $|\nodeso|=3|\nodes|/4$, \ $|\nodesD|=|\nodes|/2$, and $|\nodesS|=|\nodes|/2$. 
The lower-level cost vectors are chosen to be affine of the form 
{$[\mybalpha(\bvaflo)]_{ij} = \myalpha{ij}{1}\cdot \bvaflo_{ij} + \myalpha{ij}{0}\;\; \forall (i,j)\in \arcs$}, 
{$[\mybbeta(\bvadem)]_{i} = -\mybeta{i}{1}\cdot \bvadem_i + \mybeta{i}{0}\;\; \forall i\in \nodesD$}, 
{$[\mybgamma(\bvasup)]_{j} = \mygamma{j}{1}\cdot \bvasup_j + \mygamma{j}{0}\;\; \forall j\in \nodesS$},
where $\myalpha{ij}{1}>0$, $\mybeta{i}{1}>0$, and $\mygamma{j}{1}>0$ so that Assumptions \ref{assump:JacobSymm}-\ref{assump:Unique} are satisfied. 
Since $\brhsB=\bm{0}$ in EFL,  Lemma~\ref{lemma:BndRelaxII} establishes that the root relaxation of the {Duality-based formulation} is bounded and amounts to solving a strictly convex quadratic program. 
Define $\Unif(a,b)$ to be the continuous uniform distribution over the interval $(a,b)$. 
For each instance, the lower-level cost parameters are generated as $\myalpha{ij}{0}\sim\Unif(0,3)$, \ $\myalpha{ij}{1}\sim\Unif(0,2)$, \ $\mybeta{i}{0}\sim\Unif(1300,1500)$, \ 
$\mybeta{i}{1}\sim\Unif(3,4)$, \ $\mygamma{j}{0}\sim\Unif(1,2)$, \ and $\mygamma{j}{1}\sim\Unif(0,1)$.
The upper-level cost parameters are generated as $\conewz_i\sim\Unif(150,200)$ and $v_i \sim \Unif(3,5)$ whereas the capacity parameters are generated as $\ubvaq_i\sim\Unif(100,200)$ and $\vaqmax=350\cdot|\nodeso|/4$.

The performance of the two single-level formulations of EFL is evaluated on medium-sized and large-sized networks in Sections~\ref{sec:EFL-ComputResults1} and \ref{sec:EFL-ComputResults2}, respectively. 
A heuristic approach is described in Section~\ref{sec:EFL-ComputResults2} to aid branch-and-bound in solving large-sized network instances. 
All models are written in Python 3.7 and solved using GUROBI (v9.5.2) with parameters \code{TimeLimit=600} (for medium-sized networks),\ \code{TimeLimit=1200} (for large-sized networks),\ \code{MIPGap=0.01\%},\ \code{IntFeasTol=1e-9},\ \code{DualReductions=0}, and  \code{NonConvex=2}. 
The experiments are run on an Intel(R) Core(TM) i7-10510U CPU @ 1.80GHz machine with 16GB RAM. 

%%%%%%%%%%%%%%%%%
\subsubsection{Medium-sized networks}
\label{sec:EFL-ComputResults1}
%%%%%%%%%%%%%%%%%

In this section, we compare {KKT and Duality-based formulations} on networks with $|\nodes|\in\{10,20,30,40\}$ and $|\arcs|\in\{15,25,35,45,55,65\}$. 
For each $(|\nodes|,|\arcs|)$ pair, five random instances of network, cost, and capacity parameters are generated. 
Table~\ref{tab:EFL-B&BSummary-Results1} summarizes the results, where \textbf{I} is the instance number, \textbf{T} is the solution time (in seconds), 
\textbf{\#N} is the number of nodes explored during branch-and-bound,  \textbf{ObjVal} is the best objective value found, 
{\textbf{RootRelax} is the relaxation bound at the root node, and} \textbf{ObjBnd} is the best relaxation bound.
The entries marked as ``--"  in Table~\ref{tab:EFL-B&BSummary-Results1} denote cases where GUROBI fails to guarantee a finite upper bound on the problem within the time limit whereas {\code{UNBD} represents that the root relaxation was found to be unbounded}. 
We make the following observations:
\begin{enumerate}
    \item 
    In the time limit of 600 seconds, {the KKT-based formulation} only solves instances with approximately 10 to 20 nodes and 35 to 45 arcs whereas {the Duality-based formulation} solves all instances within the optimality gap tolerance of 0.01\%.
    \item 
    When considering only instances successfully solved by {the KKT-based formulation}, the runtime and explored nodes count is several orders of magnitude higher as compared to that of {the Duality-based formulation}. 
    In fact, {the Duality-based formulation} solves most instances at the root node and all of them within 1 second of computation time.
\end{enumerate}
{
We conjecture that the result of Lemma~\ref{lemma:Unbd_ReformI} holds for the EFL constraints {\eqref{constr:EFL-MPCC-Leader2}-\eqref{constr:EFL-EqbmPrice}} stated in Appendix~\ref{sec:EFL-Appendix_SingleLevel}, which explains the unboundedness of the root relaxation of {the KKT-based formulation} observed in Table~\ref{tab:EFL-B&BSummary-Results1}, and results in severe computational limitations, even for medium-sized networks. 
As already mentioned, the root relaxation of {the Duality-based formulation} is bounded by Lemma~\ref{lemma:BndRelaxII}.
} 
Hence, we restrict our attention to {the Duality-based formulation} for large-sized networks.

\begin{table}[!htbp]
  \centering
%  \vspace*{-1.2cm}
%  \hspace*{-1.4cm}
  \resizebox{1\textwidth}{!}{%
    \begin{tabular}{ccc|rrrrr|rrrrc}
          &       &       & 
          \multicolumn{5}{c|}{\multirow{2}{*}{\textbf{{KKT-based} formulation}}} 
          & \multicolumn{5}{c}{\multirow{2}{*}{\textbf{{Duality-based} formulation}}} \\ 
          &       &       & \multicolumn{5}{c|}{}    & \multicolumn{5}{c}{} \\
     \multirow{2}{*}{$|\nodes|$} & \multirow{2}{*}{$|\arcs|$} & \multirow{2}{*}{\textbf{I}} & \multirow{2}{*}{\textbf{T}} & \multirow{2}{*}{\textbf{ObjVal}} & \multirow{2}{*}{\textbf{ObjBnd}} & \multirow{2}{*}{{\textbf{RootRelax}}} & \multirow{2}{*}{\textbf{\#N}} & \multirow{2}{*}{\textbf{T}} & \multirow{2}{*}{\textbf{ObjVal}} & \multirow{2}{*}{\textbf{ObjBnd}} & \multirow{2}{*}{{\textbf{RootRelax}}} & \multirow{2}{*}{{\textbf{\#N}}} \\
    &       &       &       &       &       &       &       &       &       &       &       &  \\
    \hline 
    \addlinespace[1ex]
    \multirow{5}[1]{*}{10} & \multirow{5}[1]{*}{15} & 1     & 0.95  &      30,196.00  &      30,196.28  & \multicolumn{1}{c}{\code{UNBD}}  & 9.1E+03 & 0.03  &          30,196.00  &          30,198.82  &          30,557.77  & 1 \\
          &       & 2     & 0.20  &    189,463.31  &    189,463.31  & \multicolumn{1}{c}{\code{UNBD}}  & 1.5E+03 & 0.02  &        189,463.31  &        189,463.31  &        189,530.00  & 1 \\
          &       & 3     & 1.09  &    128,572.14  &    128,572.49  & \multicolumn{1}{c}{\code{UNBD}}  & 6.5E+03 & 0.03  &        128,572.14  &        128,572.14  &        130,609.20  & 1 \\
          &       & 4     & 1.56  &    196,712.65  &    196,712.85  & \multicolumn{1}{c}{\code{UNBD}}  & 8.7E+03 & 0.01  &        196,712.65  &        196,712.90  &        196,984.10  & 1 \\
          &       & 5     & 2.63  &    164,950.80  &    164,951.00  & \multicolumn{1}{c}{\code{UNBD}}  & 1.7E+04 & 0.02  &        164,950.70  &        164,950.80  &        165,166.30  & 1 \\
    \hline 
    \addlinespace[1ex]
    \multirow{5}[1]{*}{10} & \multirow{5}[1]{*}{25} & 1     & 10.62 &      38,675.81  &      38,675.81  & \multicolumn{1}{c}{\code{UNBD}}  & 5.4E+04 & 0.04  &          38,675.81  &          38,676.11  &          38,956.68  & 1 \\
          &       & 2     & 2.75  &    123,010.65  &    123,011.04  & \multicolumn{1}{c}{\code{UNBD}}  & 2.0E+04 & 0.04  &        123,010.65  &        123,019.24  &        123,560.10  & 1 \\
          &       & 3     & 15.39 &      97,375.16  &      97,379.06  & \multicolumn{1}{c}{\code{UNBD}}  & 4.5E+04 & 0.03  &          97,375.18  &          97,375.18  &          97,669.19  & 1 \\
          &       & 4     & 10.52 &    120,415.06  &    120,415.72  & \multicolumn{1}{c}{\code{UNBD}}  & 3.5E+04 & 0.02  &        120,415.06  &        120,416.22  &        120,649.40  & 1 \\
          &       & 5     & 9.86  &    111,571.93  &    111,574.53  & \multicolumn{1}{c}{\code{UNBD}}  & 4.7E+04 & 0.11  &        111,571.93  &        111,571.93  &        112,314.60  & 1 \\
   \hline 
   \addlinespace[1ex]
    \multirow{5}[1]{*}{10} & \multirow{5}[1]{*}{35} & 1     & 205.47 &      37,636.65  &      37,636.65  & \multicolumn{1}{c}{\code{UNBD}}  & 5.3E+05 & 0.08  &          37,636.65  &          37,636.65  &          37,921.54  & 1 \\
          &       & 2     & 24.85 &      95,544.40  &      95,549.08  & \multicolumn{1}{c}{\code{UNBD}}  & 1.2E+05 & 0.12  &          95,544.56  &          95,544.56  &          96,114.82  & 47 \\
          &       & 3     & 215.47 &      88,376.93  &      88,378.22  & \multicolumn{1}{c}{\code{UNBD}}  & 6.1E+05 & 0.09  &          88,376.94  &          88,377.46  &          88,841.80  & 1 \\
          &       & 4     & 195.94 &      93,342.01  &      93,342.01  & \multicolumn{1}{c}{\code{UNBD}}  & 6.2E+05 & 0.11  &          93,342.01  &          93,342.29  &          93,799.90  & 1 \\
          &       & 5     & 121.67 &    107,445.38  &    107,445.47  & \multicolumn{1}{c}{\code{UNBD}}  & 3.6E+05 & 0.07  &        107,445.38  &        107,447.96  &        108,080.90  & 1 \\
   \hline 
    \addlinespace[1ex]    
    \multirow{5}[1]{*}{20} & \multirow{5}[1]{*}{25}     & 1     & 19.54 &    498,845.56  &    498,845.68  & \multicolumn{1}{c}{\code{UNBD}}  & 1.1E+05 & 0.17  &        498,845.56  &        498,845.56  &        503,289.60  & 47 \\
          &       & 2     & 0.22  &    433,849.50  &    433,850.56  & \multicolumn{1}{c}{\code{UNBD}}  & 2.3E+03 & 0.07  &        433,849.50  &        433,852.40  &        434,171.40  & 1 \\
          &       & 3     & 2.45  &    504,330.80  &    504,332.56  & \multicolumn{1}{c}{\code{UNBD}}  & 2.0E+04 & 0.03  &        504,330.80  &        504,341.46  &        504,467.40  & 1 \\
          &       & 4     & 11.47 &    509,104.85  &    509,105.03  & \multicolumn{1}{c}{\code{UNBD}}  & 5.8E+04 & 0.11  &        509,104.85  &        509,131.45  &        509,685.00  & 1 \\
          &       & 5     & 6.35  &    499,875.99  &    499,885.33  & \multicolumn{1}{c}{\code{UNBD}}  & 4.7E+04 & 0.02  &        499,875.99  &        499,920.72  &        500,034.70  & 1 \\
    \hline 
    \addlinespace[1ex]   
    \multirow{5}[1]{*}{20} & \multirow{5}[1]{*}{35} & 1     & 17.23 &    203,934.06  &    203,936.43  & \multicolumn{1}{c}{\code{UNBD}}  & 8.1E+04 & 0.08  &        203,934.06  &        203,934.06  &        204,319.10  & 1 \\
          &       & 2     & 7.86  &    487,560.62  &    487,560.84  & \multicolumn{1}{c}{\code{UNBD}}  & 3.5E+04 & 0.05  &        487,514.05  &        487,560.62  &        489,254.00  & 1 \\
          &       & 3     & 13.87 &    306,256.69  &    306,260.87  & \multicolumn{1}{c}{\code{UNBD}}  & 6.9E+04 & 0.05  &        306,256.70  &        306,284.36  &        307,085.20  & 1 \\
          &       & 4     & 368.39 &    407,975.87  &    407,999.58  & \multicolumn{1}{c}{\code{UNBD}}  & 1.2E+06 & 0.06  &        407,975.87  &        407,978.52  &        408,289.50  & 1 \\
          &       & 5     & 66.09 &    327,594.00  &    327,615.84  & \multicolumn{1}{c}{\code{UNBD}}  & 3.0E+05 & 0.27  &        327,594.29  &        327,594.29  &        329,018.50  & 50 \\
     \hline 
    \addlinespace[1ex]  
    \multirow{5}[1]{*}{20} & \multirow{5}[1]{*}{45} & 1     & 311.67 &    142,070.13  &    142,074.87  & \multicolumn{1}{c}{\code{UNBD}}  & 1.2E+06 & 0.12  &        142,070.14  &        142,070.14  &        142,552.10  & 1 \\
          &       & 2     & 176.98 &    449,096.58  &    449,096.58  & \multicolumn{1}{c}{\code{UNBD}}  & 3.9E+05 & 0.07  &        449,096.58  &        449,096.58  &        460,733.40  & 1 \\
          &       & 3     & 600.02 &    282,641.07  &  \multicolumn{1}{c}{--}  & \multicolumn{1}{c}{\code{UNBD}}  & 1.2E+06 & 0.02  &        282,641.07  &        282,663.94  &        284,576.60  & 1 \\
          &       & 4     & 600.01 &    406,302.66  &  \multicolumn{1}{c}{--}  & \multicolumn{1}{c}{\code{UNBD}}  & 1.6E+06 & 0.34  &        407,886.81  &        407,906.37  &        408,920.40  & 1 \\
          &       & 5     & 381.94 &    318,349.93  &    318,379.28  & \multicolumn{1}{c}{\code{UNBD}}  & 1.1E+06 & 0.39  &        318,349.99  &        318,351.48  &        319,390.80  & 1 \\
    \hline 
    \addlinespace[1ex] 
    \multirow{5}[1]{*}{30} & \multirow{5}[1]{*}{35} & 1     & 600.01 &    710,038.90  &  \multicolumn{1}{c}{--}  & \multicolumn{1}{c}{\code{UNBD}}  & 2.5E+06 & 0.12  &        714,821.96  &        714,870.07  &        715,313.90  & 1 \\
          &       & 2     & 600.02 &      68,847.05  &  \multicolumn{1}{c}{--}  & \multicolumn{1}{c}{\code{UNBD}}  & 2.3E+06 & 0.28  &        201,915.23  &        201,915.50  &        204,304.70  & 1 \\
          &       & 3     & 497.01 &    662,744.86  &    662,796.14  & \multicolumn{1}{c}{\code{UNBD}}  & 1.8E+06 & 0.09  &        662,744.89  &        662,795.59  &        664,202.50  & 1 \\
          &       & 4     & 600.03 &    461,753.24  &  \multicolumn{1}{c}{--}  & \multicolumn{1}{c}{\code{UNBD}}  & 1.9E+06 & 0.10  &        609,964.41  &        609,964.41  &        611,133.60  & 1 \\
          &       & 5     & 271.54 &    863,008.81  &    863,014.06  & \multicolumn{1}{c}{\code{UNBD}}  & 7.7E+05 & 0.11  &        863,008.81  &        863,008.81  &        863,774.20  & 1 \\
    \hline 
    \addlinespace[1ex]  
    \multirow{5}[1]{*}{30} & \multirow{5}[1]{*}{45} & 1     & 600.03 &    664,063.25  &  \multicolumn{1}{c}{--}  & \multicolumn{1}{c}{\code{UNBD}}  & 1.6E+06 & 0.15  &        665,343.64  &        665,343.64  &        666,448.80  & 1 \\
          &       & 2     & 600.03 &    141,248.40  &  \multicolumn{1}{c}{--}  & \multicolumn{1}{c}{\code{UNBD}}  & 1.4E+06 & 0.20  &        299,694.53  &        299,694.53  &        301,966.10  & 30 \\
          &       & 3     & 600.03 &    235,747.70  &  \multicolumn{1}{c}{--}  & \multicolumn{1}{c}{\code{UNBD}}  & 1.6E+06 & 0.06  &        546,912.17  &        546,912.17  &        547,887.40  & 1 \\
          &       & 4     & 600.03 &    420,121.42  &  \multicolumn{1}{c}{--}  & \multicolumn{1}{c}{\code{UNBD}}  & 1.6E+06 & 0.22  &        585,424.48  &        585,424.48  &        586,656.40  & 130 \\
          &       & 5     & 600.01 &    663,150.48  &  \multicolumn{1}{c}{--}  & \multicolumn{1}{c}{\code{UNBD}}  & 1.3E+06 & 0.06  &        663,784.20  &        663,843.67  &        663,973.30  & 1 \\
    \hline 
    \addlinespace[1ex]  
    \multirow{5}[1]{*}{30} & \multirow{5}[1]{*}{55} & 1     & 600.04 &    428,639.53  &  \multicolumn{1}{c}{--}  & \multicolumn{1}{c}{\code{UNBD}}  & 1.5E+06 & 0.30  &        593,085.72  &        593,085.72  &        594,287.40  & 1 \\
          &       & 2     & 600.03 &    284,773.92  &  \multicolumn{1}{c}{--}  & \multicolumn{1}{c}{\code{UNBD}}  & 7.1E+05 & 0.68  &        291,639.64  &        291,658.74  &        294,233.00  & 1 \\
          &       & 3     & 600.03 &    364,145.32  &    399,681.36  & \multicolumn{1}{c}{\code{UNBD}}  & 1.8E+06 & 0.07  &        364,183.87  &        364,211.77  &        364,670.50  & 1 \\
          &       & 4     & 600.03 &    447,788.67  &  \multicolumn{1}{c}{--}  & \multicolumn{1}{c}{\code{UNBD}}  & 6.7E+05 & 0.24  &        457,940.33  &        457,940.33  &        459,872.10  & 120 \\
          &       & 5     & 600.01 &  \multicolumn{1}{c}{--}  &  \multicolumn{1}{c}{--}  & \multicolumn{1}{c}{\code{UNBD}}  & 1.1E+06 & 0.21  &        599,448.87  &        599,450.31  &        599,904.00  & 1 \\
    \hline 
    \addlinespace[1ex]
    \multirow{5}[1]{*}{40} & \multirow{5}[1]{*}{45} & 1     & 600.02 &    499,194.33  &  \multicolumn{1}{c}{--}  & \multicolumn{1}{c}{\code{UNBD}}  & 2.1E+06 & 0.07  &        878,326.04  &        878,333.86  &        878,761.70  & 1 \\
          &       & 2     & 600.04 &    329,915.35  &  \multicolumn{1}{c}{--}  & \multicolumn{1}{c}{\code{UNBD}}  & 1.7E+06 & 0.07  &        987,269.65  &        987,340.09  &        989,087.40  & 1 \\
          &       & 3     & 600.02 &    943,228.53  &  \multicolumn{1}{c}{--}  & \multicolumn{1}{c}{\code{UNBD}}  & 1.9E+06 & 0.11  &    1,150,729.89  &    1,150,731.25  &    1,151,642.00  & 1 \\
          &       & 4     & 600.03 &    366,596.14  &  \multicolumn{1}{c}{--}  & \multicolumn{1}{c}{\code{UNBD}}  & 2.0E+06 & 0.09  &    1,268,511.16  &    1,268,553.33  &    1,269,592.00  & 1 \\
          &       & 5     & 600.03 &    226,573.16  &  \multicolumn{1}{c}{--}  & \multicolumn{1}{c}{\code{UNBD}}  & 1.7E+06 & 0.21  &        710,472.47  &        710,472.47  &        722,389.30  & 1 \\
    \hline 
    \addlinespace[1ex] 
    \multirow{5}[1]{*}{40} & \multirow{5}[1]{*}{55} & 1     & 600.01 &    689,235.43  &  \multicolumn{1}{c}{--}  & \multicolumn{1}{c}{\code{UNBD}}  & 2.1E+06 & 0.20  &        873,272.97  &        873,310.71  &        874,346.50  & 1 \\
          &       & 2     & 600.02 &    354,451.47  &  \multicolumn{1}{c}{--}  & \multicolumn{1}{c}{\code{UNBD}}  & 1.4E+06 & 0.18  &        899,208.53  &        899,273.88  &        900,978.70  & 1 \\
          &       & 3     & 600.02 &    555,781.08  &  \multicolumn{1}{c}{--}  & \multicolumn{1}{c}{\code{UNBD}}  & 1.0E+06 & 0.21  &    1,059,822.79  &    1,059,897.36  &    1,061,789.00  & 1 \\
          &       & 4     & 600.04 &    465,578.18  &  \multicolumn{1}{c}{--}  & \multicolumn{1}{c}{\code{UNBD}}  & 1.6E+06 & 0.04  &    1,168,314.25  &    1,168,325.19  &    1,168,685.00  & 1 \\
          &       & 5     & 600.03 &    389,722.74  &  \multicolumn{1}{c}{--}  & \multicolumn{1}{c}{\code{UNBD}}  & 1.7E+06 & 0.07  &        711,535.60  &        711,537.96  &        723,845.50  & 36 \\     
    \hline 
    \addlinespace[1ex]     
    \multirow{5}[1]{*}{40} & \multirow{5}[1]{*}{65} & 1     & 600.02 &    394,876.13  &  \multicolumn{1}{c}{--}  & \multicolumn{1}{c}{\code{UNBD}}  & 1.6E+06 & 0.15  &        883,248.41  &        883,255.85  &        885,437.30  & 1 \\
          &       & 2     & 600.02 &    314,351.29  &  \multicolumn{1}{c}{--}  & \multicolumn{1}{c}{\code{UNBD}}  & 1.4E+06 & 0.22  &        899,579.67  &        899,607.05  &        901,270.80  & 57 \\
          &       & 3     & 600.01 &    566,669.73  &  \multicolumn{1}{c}{--}  & \multicolumn{1}{c}{\code{UNBD}}  & 1.3E+06 & 0.52  &        924,406.16  &        924,492.16  &        926,630.10  & 1 \\
          &       & 4     & 600.03 &    701,145.59  &  \multicolumn{1}{c}{--}  & \multicolumn{1}{c}{\code{UNBD}}  & 1.9E+06 & 0.11  &    1,139,441.56  &    1,139,441.56  &    1,139,999.00  & 1 \\
          &       & 5     & 600.02 &      92,212.37  &  \multicolumn{1}{c}{--}  & \multicolumn{1}{c}{\code{UNBD}}  & 9.7E+05 & 0.44  &        570,315.76  &        570,319.21  &        571,631.70  & 7 \\
    \hline      
    \end{tabular}%
    }
  \captionsetup{justification=centering}
  \caption{\scriptsize \textbf{{KKT-based vs Duality-based formulations}} on medium-sized EFL instances.}
  \label{tab:EFL-B&BSummary-Results1}
\end{table}%

%%%%%%%%%%%%%%%%%
\subsubsection{Large-sized networks} 
\label{sec:EFL-ComputResults2}
%%%%%%%%%%%%%%%%%

We conduct another set of experiments on networks with 100 nodes, \ie $|\nodes|=100$, and with a varying number $|\arcs|$ of arcs. 
The aim is to investigate the performance of the stronger {Duality-based} formulation on large-sized networks. 
Only one instance is considered for each value of $|\arcs|$. 
Specifically, $|\arcs|$ is gradually increased in steps of 470 by randomly adding new arcs to the previous set of arcs while keeping all other parameters unchanged. 
The results are summarized in the first half of Table~\ref{tab:EFL-B&BSummary-Results2}. {In Table~\ref{tab:EFL-B&BSummary-Results2}, \textbf{\%Gap} is the optimality gap at termination}. 

With {the Duality-based formulation}, GUROBI can handle instances with fewer than 3400 arcs but struggles to find a lower bound (\ie to find a feasible solution) for instances having 3800 arcs or more. 
The bounded objective of {the Duality-based formulation} at the root node ensures that an upper bound is found for all instances {when the root relaxation is solved within the time limit}.

We develop a simple rounding heuristic (\RH) that can be called during branch-and-bound for finding good quality feasible solutions. 
A pseudo-code of this procedure is given in Appendix \ref{sec:EFL-Appendix_RH}; see Algorithm~\ref{algo:roundingI}.
Given a fractional solution $(\widehat \bvaz, \widehat \bvaq)$ available at any point of the branch-and-bound procedure, this heuristic rounds the components of $\widehat \bvaz$ larger than rounding threshold $\RndThRH$ to $1$, sets the others to zeros, and decreases the production quantities of facilities that were just closed to $0$. 
Clearly, the solution $(\mymark \bvaz, \mymark \bvaq)$ so obtained satisfies the upper level constraints $\feas_\textrm{EFL}$.
For this vector of leader variables, the heuristic then solves the follower problem. 
Because we selected affine lower-level costs in our instances, the primal and dual problems are convex quadratic programs that can be efficiently solved using GUROBI;  see \eqref{EFL:Follower_primal} and \eqref{EFL:Follower_dual} in Appendix~\ref{sec:EFL-Appendix_FollowerPrimalDual}.

We invoke \RH at every node of the branch-and-bound tree until the point where two feasible solutions have been generated. 
After that point, \RH is run at a node only if there is a successful trial of a Bernoulli distribution with success probability 5\%. 

The second half of Table~\ref{tab:EFL-B&BSummary-Results2} summarizes the results for {the Duality-based formulation} when \RH is utilized. 
In this table, \textbf{\#\RH Execs} is the number of times \RH is called during the branch-and-bound search and \textbf{Time/\RH Exec} is the average time required per execution of \RH. 
We make the following observations:
\begin{enumerate}
    \item 
    Using \RH, {the Duality-based formulation} solves all instances within an optimality gap of 1\%, but most of them reach the time limit of 1200 seconds.
    \item 
    As the number of arcs increases, the size of the follower primal-dual problems becomes larger, which results in an increase in the time spent per \RH execution and in a decrease of the number of \RH executions during the allotted time.
    \item 
    Fewer nodes are explored when using \RH as compared to when \RH is not used. 
    Moreover, as the number of arcs in a network increases for given cost parameters, the optimal value decreases. 
    The reason is the increase in number of constraints \eqref{constr:EFL-EqbmPrice} on the equilibrium price vector $\bdvapi$.
\end{enumerate}

\begin{sidewaystable}[!htbp]
  \centering
  % \hspace*{-0.25cm}
  \resizebox{1\textwidth}{!}{%
    \begin{tabular}{c|rrrrr|rrrrr|rr}
     & \multicolumn{5}{c|}{\textbf{{Duality-based} formulation}} 
     & \multicolumn{7}{c}{\textbf{{Duality-based} formulation} + \RH [$\RndThRH=0.5$]} \\      
     & \multicolumn{5}{c|}{} & \multicolumn{7}{c}{} \\ 
    \multirow{2}{*}{$|\arcs|$} & \multirow{2}{*}{\textbf{T}} & \multirow{2}{*}{\textbf{ObjVal}} & \multirow{2}{*}{\textbf{ObjBnd}} & \multirow{2}{*}{\%\textbf{Gap}} & \multirow{2}{*}{\textbf{\#N}} & \multirow{2}{*}{\textbf{T}} & \multirow{2}{*}{\textbf{ObjVal}} & \multirow{2}{*}{\textbf{ObjBnd}} & \multirow{2}{*}{\%\textbf{Gap}} & \multirow{2}{*}{\textbf{\#N}} & {\textbf{\# RH}} & {\textbf{Time/}}\\
        &       &       &       &       &   \textbf{}   &       &       &       &       &   \textbf{} & {\textbf{Execs}} & {\textbf{RH Exec}}  \\ \hline
          &       &       &       &       &       &       &       &       &       &       &       &  \\
    100   & 0.15  &           2,119,361.60  &           2,119,361.60  & 0.000 &                          1  & 0.16  &     2,119,361.60  &           2,119,361.60  & 0.000 &                   1  & 0     &  \\
    570   & 10.76 &              650,467.42  &              650,507.78  & 0.006 &                 6,115  & 80.30 &         650,466.70  &              650,508.41  & 0.006 &           6,564  & 217   & 0.25 \\
    1040  & 724.72 &              506,984.82  &              507,032.81  & 0.009 &             265,152  & 1208.30 &         506,979.30  &              507,307.56  & 0.065 &         63,133  & 2102  & 0.47 \\
    1510  & 347.67 &              444,830.62  &              444,867.43  & 0.008 &               26,023  & 911.58 &         444,836.95  &              444,874.14  & 0.008 &         23,690  & 732   & 0.79 \\
    1980  & 392.08 &              410,665.20  &              410,693.85  & 0.007 &               52,259  & 1219.80 &         410,665.42  &              410,855.78  & 0.046 &         23,913  & 762   & 1.15 \\
    2450  & 502.07 &              389,643.87  &              389,676.47  & 0.008 &               34,947  & 1229.78 &         389,643.87  &              389,845.14  & 0.052 &         17,861  & 556   & 1.59 \\
    2920  & 1200.39 &              363,512.99  &              363,738.14  & 0.062 &               73,062  & 1203.87 &         363,476.39  &              363,932.78  & 0.126 &         10,865  & 375   & 2.05 \\
    3390  & 1200.07 &              354,221.20  &              354,621.21  & 0.113 &               21,565  & 1215.40 &         354,251.00  &              354,617.54  & 0.103 &           8,750  & 306   & 2.45 \\
    3860  & 1200.06 &  \multicolumn{1}{c}{--}  &              347,199.30  & \multicolumn{1}{c}{--} &               10,552  & 1211.74 &         346,565.59  &              346,876.96  & 0.090 &           6,687  & 232   & 3.01 \\
    4330  & 1200.06 &  \multicolumn{1}{c}{--}  &              330,324.71  & \multicolumn{1}{c}{--} &               10,678  & 1231.43 &         329,445.52  &              330,074.57  & 0.191 &           5,594  & 201   & 3.61 \\
    4800  & 1200.07 &  \multicolumn{1}{c}{--}  &              322,977.84  & \multicolumn{1}{c}{--} &                 8,662  & 1200.08 &         322,009.42  &              322,916.25  & 0.282 &           3,110  & 123   & 4.06 \\
    5270  & 1200.03 &  \multicolumn{1}{c}{--}  &              314,437.13  & \multicolumn{1}{c}{--} &               10,155  & 1200.75 &         313,480.74  &              314,255.57  & 0.247 &           4,315  & 163   & 4.88 \\
    5740  & 1200.09 &  \multicolumn{1}{c}{--}  &              311,383.15  & \multicolumn{1}{c}{--} &                 8,725  & 1206.90 &         310,658.12  &              311,220.86  & 0.181 &           3,093  & 116   & 5.64 \\
    6210  & 1201.48 &  \multicolumn{1}{c}{--}  &              307,847.91  & \multicolumn{1}{c}{--} &                 5,127  & 1204.45 &         306,844.53  &              307,817.80  & 0.317 &           1,771  & 79    & 6.30 \\
    6680  & 1200.04 &  \multicolumn{1}{c}{--}  &              303,935.43  & \multicolumn{1}{c}{--} &                 5,180  & 1200.08 &         303,152.01  &              303,891.56  & 0.244 &           2,239  & 87    & 6.95 \\
    7150  & 1200.16 &  \multicolumn{1}{c}{--}  &              301,774.24  & \multicolumn{1}{c}{--} &                 5,359  & 1226.24 &         299,771.08  &              301,833.72  & 0.688 &           2,588  & 109   & 6.39 \\
    7620  & 1200.07 &  \multicolumn{1}{c}{--}  &              299,951.35  & \multicolumn{1}{c}{--} &                 5,257  & 1206.85 &         299,034.18  &              299,999.50  & 0.323 &           2,504  & 102   & 7.41 \\
    8090  & 1200.18 &  \multicolumn{1}{c}{--}  &              298,796.14  & \multicolumn{1}{c}{--} &                 5,210  & 1200.03 &         297,932.89  &              298,772.81  & 0.282 &           2,083  & 87    & 8.51 \\
    8560  & 1200.08 &  \multicolumn{1}{c}{--}  &              298,067.73  & \multicolumn{1}{c}{--} &                 5,199  & 1204.08 &         296,417.40  &              298,123.74  & 0.576 &           1,549  & 76    & 9.36 \\
    9030  & 1200.05 &  \multicolumn{1}{c}{--}  &              296,944.87  & \multicolumn{1}{c}{--} &                 4,677  & 1200.04 &         296,290.47  &              296,957.08  & 0.225 &           1,845  & 77    & 10.38 \\ 
    \end{tabular}%
    }
  \caption{\scriptsize \textbf{{Duality-based} formulation}  with and without Rounding Heuristic (\textbf{RH}) on large-sized EFL instances with $|\nodes|=100$.}
  \label{tab:EFL-B&BSummary-Results2}%
\end{sidewaystable}%

%%%%%%%%%%%%%%%%%%%%%%%%%%%%%%%%%%
\subsection{Application 2: {RGUP in power distribution networks}} % Planning of renewable generation units
\label{sec:RGU-Descrp}
%%%%%%%%%%%%%%%%%%%%%%%%%%%%%%%%%%

Due to rising electricity demand, power grids are often burdened with very large loads, which may result in power outages in the worst case. 
A possible solution is to integrate renewable generation units (RGUs) into power distribution networks to improve reliability. 
As a result, the optimal deployment of RGUs in distribution networks has attracted recent attention from the research community; see \cite{zhang2017optimal}. 
In a power distribution network, there are several firms, each controlling a number of generating units. 
Each generation unit submits a bid to the independent system operator (ISO). 
This bid defines the supply-price curve at each of the supply nodes. 
ISO then decides how much power to buy from the different units, how much to deliver to consumers, and what prices to charge based on the solution of an Optimal Power Flow (OPF) problem. 

Consider a power distribution system represented as a directed network $G=(\nodes,\arcs)$ \big(where $\nodes$ is the set of nodes/buses and $\arcs$ is the set of arcs/lines\big) with demand nodes $\nodesD\subseteq \nodes$ and supply nodes $\nodesS\subseteq \nodes$. 
Define $\nodeso$ to be the set of nodes/buses under control of the leader firm where RGUs with capacity $\bvaq_i$ for $i\in \nodeso$ can be located. {The upper-level feasible set is $\feas_\textrm{RGUP} = 
\left\{
 (\bvaz,\bvaq) \in \Re^{2|\nodeso|} \ \big|\  
 0 \leq \bvaq_i \leq \ubvaq_i  \bvaz_i \; \forall i\in \nodeso, \;
 \bvaz \in \{0,1\}^{|\nodeso|}
 \right\}
$
where $\bvaz_i$ is the RGU location decision and $\ubvaq_i$ is the maximum possible RGU capacity at node $i\in\nodeso$.}
We assume that there is one generation unit per node which means $\nodesS\cap\nodeso=\emptyset$ and that ISO accepts all the RGU generation so that there is no bidding for nodes in $\nodeso$. 
The {lower level} is the OPF problem faced by the ISO, which can be understood as a single commodity SPE problem (with additional constraints due to Kirchoff's voltage law) \cite{hobbs2000strategic} where the supply curves are the bid submitted by generation units at nodes $\nodesS$. 
To simplify the derivations, we consider a DC OPF model where resistance is assumed negligible relative to reactance and is ignored. 
Denote 
$\bvaflo=(\bvaflo_{ij},\ (i,j)\in\arcs)$, 
$\bvadem=(\bvadem_i,\ i \in \nodesD)$, 
$\bvasup=(\bvasup_j,\ j \in \nodesS)$ to be the vectors of power flows, demands, and supplies in the network, respectively. 
Assuming competition within the network, the installation of RGU capacity $\bvaq$ increases total power generation capacity and impacts equilibrium prices and power flows in the distribution network. 
We also consider uncertainty in RGU generation \cite{zhang2017optimal} using $\samp=(\samp_i,\ i \in \nodeso)$ where each $0\leq\samp_i\leq 1$. Here, $\samp_i$ is the fraction of capacity $\bvaq_i$ that is realized into actual RGU generation.
For a given $\samp$, the resulting power flows, demand, and supply will produce a new competitive equilibrium {$(\bvadem^*,\bvasup^*)$}, which is obtained by solving the variational inequality
\begin{align*}
    - \langle\mybbeta(\bvadem^*),\bvadem-\bvadem^*\rangle + \langle \mybgamma(\bvasup^*), \bvasup-\bvasup^*\rangle \geq 0, \quad \forall (\bvadem, \bvasup) \in \proj_{\bvadem,\bvasup} \myOmega(\bvaq,\samp),
\end{align*}
where $\mybbeta(\bvadem)$ and $\mybgamma(\bvasup)$ are the inverse demand cost vector and supply bid functions, respectively. 
The set $\myOmega(\bvaq,\samp)$ corresponds to the power flow balance, Kirchoff's voltage law \cite{hobbs2000strategic} together with line and generation capacity constraints given as
\begin{eqnarray*}
    \Omega(\bvaq,\samp) = \left\{
    \begin{array}{c|c}
    (\bvaflo,\bvadem,\bvasup)\; & 
    \begin{array}{ll}
    & \idyof{} \bvaflo + \idyoD{} \bvadem - \Diag(\samp) \bvaq = 0,\; \\
    & \idyif{} \bvaflo + \idyiD{} \bvadem - \idyiS{} \bvasup = 0, \\
    & \reac \bvaflo = 0, \\
    & 0\leq \bvaflo \leq \ubvaflo,\; 0 \leq \bvasup \leq \ubvasup,\; \bvadem \geq 0 
    \end{array}
    \end{array}
    \right\},
\end{eqnarray*}
where $\idyif{}$, $\idyiD{}$, and $\idyiS{}$ are the node-arc incidence, demand node incidence, supply node incidence matrices corresponding to the set of nodes $\nodesi=\nodes\setminus\nodeso$, respectively, and matrices $(\idyof{})_{\nodeso\times \arcs}$ and $(\idyoD{})_{\nodeso\times \nodesD}$ are defined similarly for the set $\nodeso$. 
{The definitions of these incidence matrices are given in Section~\ref{sec:EFL-Descrp}}. 
Further, $\ubvaflo$ is the vector of line capacities, $\ubvasup$ is the vector of generator capacities, and $\reac$ is the incidence matrix of signed reactance coefficients \cite{hobbs2000strategic}, \ie
     $\reac_{m,ij} = s_{ijm} r_{ij}$  if $(i,j)\in L_m$ and $0$ otherwise,   
where $m$ indexes Kirchoff voltage 
loops,
%\footnote{The Kirchoff voltage loops in an undirected network can be determined using \code{cycle\_basis()} function as part of Python package \code{NetworkX}. 
%Direction of arcs can then be used to determine their orientation in a loop.}, 
$L_m$ is the ordered set of arcs in loop $m$, $s_{ijm}=\pm1$ depending on the orientation of arc $(i,j)$ in loop $m$, and $r_{ij}$ is the reactance of line $(i,j)$. 
{
The single-level reformulation of RGUP is a two-stage stochastic program given in Appendix~\ref{sec:RGU-Appendix_SingleLevel} that maximizes the leader firm's expected profit. 
We draw finite samples for the uncertainty $\{\samp^{\langle {k}\rangle}\}_{{k}=1}^{\nsample}$, and use sample average approximation for the expectation.
The KKT-based and the Duality-based formulations of RGUP can be obtained from \eqref{eq:SL-MPCC} and \eqref{eq:SL-LagrDual} using the application specific constraints described above.
Their explicit forms can be found in Appendix~\ref{sec:RGU-Appendix_SingleLevel}.
}

%%%%%%%%%%%%%%%%%
\subsubsection{Test instances}
\label{sec:RGU-InstGen}
%%%%%%%%%%%%%%%%%

We use the standard IEEE bus systems summarized in Table~\ref{tab:IEEE-Dataset} as the power distribution networks in our numerical study. 
As is common in the power generation literature, we combine all lines between each pair of nodes in the data set into an equivalent single line. 
Each line is then transformed into a pair of opposite arcs in order to obtain a directed network which allows power flow in either direction between a pair of nodes.

The set of potential RGU locations $\nodeso$ is randomly selected from set $\nodes\setminus\nodesS$ where $|\nodeso|$ is given in Table~\ref{tab:IEEE-Dataset}. 
The first-stage cost parameters of the upper-level problem are randomly generated as $\conewz_i\sim\Unif(150,200)$ and $\coq_i \sim \Unif(3,5)$ for each $i\in\nodeso$. 
The lower-level cost vectors are chosen to be affine functions of the form:
{$[\mybbeta(\bvadem)]_{i} = -\mybeta{i}{1}\cdot \bvadem_i + \mybeta{i}{0}\;\; \forall i\in\nodesD$}, {$[\mybgamma(\bvasup)]_{j} = \mygamma{j}{1}\cdot \bvasup_j + \mygamma{j}{0}\;\; \forall j\in\nodesS$},
where $\mybeta{i}{1}>0$ and $\mygamma{j}{1}>0$ so that Assumptions \ref{assump:JacobSymm}-\ref{assump:Unique} are satisfied.
Since $\brhsB=\bm{0}$ in this application, we have from Lemma~\ref{lemma:BndRelaxII} that the root relaxation of {the Duality-based formulation} is bounded.
For a load bus $i$, the intercept parameter $\mybeta{i}{0}$ is set to $40$ and the slope $\mybeta{i}{1}$ is determined so that the resulting cost is $30$ at the rated load in megawatt (MW). 
More specifically, we set
$
\mybeta{i}{1} = \frac{40-30}{\text{Load MW rating at bus $i$}}.
$
For a generator bus $j$, the intercept and slope parameters $\mygamma{j}{0}$ and $\mygamma{j}{1}$ are fixed so that $10<\mygamma{j}{0}<33$ and $0.03<\mygamma{j}{1}<0.70$. 
The uncertainty samples are drawn 
%according to 
from a uniform distribution $\samp_i\sim \Unif(0,1)$ for $i\in\nodeso$ where the sample size $\nsample \in\{10,25,50,100\}$. 
All models are solved using GUROBI (v9.5.2) in Python 3.7 with {the} 
% same 
parameters and machine settings described in Section~\ref{sec:EFL-InstGen}. 
\begin{table}[h]
    \centering
     \vspace*{-0.1cm}
    % \hspace*{1cm}
    {
    \resizebox{0.6\textwidth}{!}{%
    \begin{tabular}{c|c|c|c|c}     
     & \multicolumn{3}{c|}{\textbf{IEEE Power Flow Test Cases}} &  \\
     & \multicolumn{3}{c|}{(\url{https://labs.ece.uw.edu/pstca/})} & \\
     \hline
    \textbf{Dataset} &  \multirow{2}{*}{\textbf{\# Lines}}  & \textbf{\# Load} & \textbf{\# Generator} & \multirow{2}{*}{$|\nodeso|$}  \\ 
    \textbf{$|\nodes|$}  &   &  \textbf{Buses}, $|\nodesD|$ & \textbf{Buses}, $|\nodesS|$ & \\ \hline 
      14 Bus    & 20  & 11 & 2 & 5 \\
     30 Bus   & 41  & 21 & 2 & 10\\
     57 Bus   & 80   & 42  & 4 & 20 \\
     118 Bus  & 186  & 91 & 19 & 40 \\
     300 Bus   & 411 & 188   & 56 & 80 \\ \hline
    \end{tabular}%
    }
    \caption{\scriptsize Standard IEEE test networks.}    
    \label{tab:IEEE-Dataset}
    }
\end{table}

%%%%%%%%%%%%%%%%%
\subsubsection{Computational results}
\label{sec:RGU-ComputResults1}
%%%%%%%%%%%%%%%%%

First we compare the {KKT and Duality-based formulations} in Table~\ref{tab:RGU-B&BSummary-Results1} on a 3 bus network with 3 lines, 2 generators, and 2 load buses. 
We fix the upper-level cost parameters to 0 and vary the sample size $\nsample$ of uncertainty from 1 to 5. 
In this case, we again conjecture that the result of Lemma \ref{lemma:Unbd_ReformI} holds for the {RGUP constraints presented in Appendix~\ref{sec:RGU-Appendix_SingleLevel}}
{as we observe that} the root relaxation of {the KKT-based formulation} is unbounded 
and limits the solution to {at most} 4 samples in a time limit of 600 seconds; see Table~\ref{tab:RGU-B&BSummary-Results1}. 
Further for $\nsample\in\{1,2,3,4\}$, {the Duality-based formulation} is much faster and solves all instances at the root node whereas {the KKT-based formulation} explores a number of nodes that is several orders of magnitude larger. 
Therefore, for the remainder of this section we focus on {the Duality-based formulation}.

\begin{table}[!htbp]
   \centering
   % \hspace*{-0.95cm}
  \resizebox{1\textwidth}{!}{%
    \begin{tabular}{cccc|rrrcr|rrrcc}
         &       &       &       & \multicolumn{5}{c|}{\multirow{2}{*}{\textbf{{KKT-based} formulation}}}
         & \multicolumn{5}{c}{\multirow{2}{*}{\textbf{{Duality-based} formulation}}} \\ 
         &       &       &       & \multicolumn{5}{c|}{}     & \multicolumn{5}{c}{} \\ %\hline
    \cmidrule{5-14} 
    \multirow{2}{*}{$|\nodes|$} & \multirow{2}{*}{$\arcs$} & \textbf{\# RGUs}, & \multirow{2}{*}{$\nsample$} & \multirow{2}{*}{\textbf{T}} & \multirow{2}{*}{\textbf{ObjVal}} & \multirow{2}{*}{\textbf{ObjBnd}} & \multirow{2}{*}{\%\textbf{Gap}} & \multirow{2}{*}{\textbf{\#N}} & \multirow{2}{*}{\textbf{T}} & \multirow{2}{*}{\textbf{ObjVal}} & \multirow{2}{*}{\textbf{ObjBnd}} & \multirow{2}{*}{\%\textbf{Gap}} & \multirow{2}{*}{\textbf{\#N}}\\
        &       &  $|\nodeso|$   &   &       &       &       &  &   &       &       &       &       &  \\ 
    \hline
          &       &       &       &       &       &       &       &       &       &       &       &       &  \\
    \multirow{5}[1]{*}{3} & \multirow{5}[1]{*}{6} & \multirow{5}[1]{*}{1} & 1     & 0.08  & 272.39 & 272.39 & 0     & 1.3E+02 & 0.01  & 272.39 & 272.39 & 0     & 1 \\
          &       &       & 2     & 0.50  & 206.75 & 206.75 & 0     & 5.6E+03 & 0.01  & 206.75 & 206.75 & 0     & 1 \\
          &       &       & 3     & 11.01 & 316.89 & 316.89 & 0     & 7.9E+04 & 0.02  & 316.89 & 316.89 & 0     & 1 \\
          &       &       & 4     & 363.72 & 436.66 & 436.66 & 0     & 1.8E+06 & 0.02  & 436.66 & 436.66 & 0     & 1 \\
          &       &       & 5     & 600.01 & 480.63 & \multicolumn{1}{c}{--} & \multicolumn{1}{c}{--} & 2.1E+06 & 0.03  & 480.63 & 480.63 & 0     & 1 \\
    \hline    
    \end{tabular}%
    }
    \captionsetup{justification=centering}
   \caption{\scriptsize \textbf{{KKT-based vs Duality-based formulations}} on 3 bus network.}
  \label{tab:RGU-B&BSummary-Results1}
\end{table}%

Second, we study the performance of {the Duality-based formulation} on standard IEEE instances and cost parameters described in Section~\ref{sec:RGU-InstGen}. 
Table~\ref{tab:RGU-B&BSummary-Results2}  summarizes the results where the last column \textbf{Avg Time RootRelax} is the average time spent in solving the root relaxation. 
For 300 bus networks and $\nsample\in\{50,100\}$, the time limit is set to 1200 seconds. 
For the remaining combinations, the time limit is 600 seconds. 
We make the following observations from the first half of Table~\ref{tab:RGU-B&BSummary-Results2}:
\begin{enumerate}
    \item 
    For 14 and 30 bus networks with sample size $\nsample\in\{10,25,50\}$, the solver is able to successfully find lower and upper bounds on the optimal value. 
    The instances, however, reach the time limit of 600 seconds and terminate with a gap larger than tolerance of 0.01\%.
    \item 
    The 57 bus network with $\nsample\in\{10,25\}$ and 118 bus network with $\nsample=10$ can be handled. 
    For larger sample sizes, however, the solver cannot find a lower bound (\ie find a feasible solution) within 600 seconds. 
    In the case of the 300 bus network, no lower bound is found for any value of $\nsample$.
    \item 
    The last column shows that the time spent in solving the root relaxation grows roughly fourfold for each twofold increase in sample size $\nsample$. 
    As the root relaxation becomes more computationally expensive with increasing $\nsample$, fewer {branch-and-bound} nodes are explored within the given time limit.
\end{enumerate}
The above observations suggest that finding a feasible solution as early as possible in the branch-and-bound tree should help solving the larger sized instances. 
Hence, we use a rounding heuristic \RH similar to that described in Section~\ref{sec:EFL-ComputResults2}; see Appendix~\ref{sec:RGU-Appendix_RH} for details. 
We select $\RndThRH=0.5$. 
Further, we invoke \RH at every node of the branch-and-bound tree until the point where one feasible solution has been generated. 
After this point, we stop running \RH. 
The results obtained after using \RH are given in Table~\ref{tab:RGU-B&BSummary-Results2} where \textbf{Time/\RH Exec} is the time spent per \RH execution. 
We make significant progress within the time limit on instances previously unsolved (except for the case with a 300 bus and with $\nsample=100$) by exploring fewer nodes for most cases. 
For instances that were solved before, \RH typically improves either the runtime (\textit{e.g.}, 57 bus with $\nsample=10$) or the optimality gap at termination (\textit{e.g.}, 14 bus with $\nsample=50$). 
We also observe that, for any given bus network, the gaps tend to increase with $\nsample$.

\begin{sidewaystable}[htbp]
  \centering
  % \hspace*{-0.15cm}
  \resizebox{1\textwidth}{!}{%
    \begin{tabular}{cccc|rrrrr|rrrrrr|c}
    &     &       &       &  \multicolumn{5}{c|}{\textbf{{Duality-based} formulation}} 
    & \multicolumn{6}{c|}{\textbf{{Duality-based} formulation } + \RH [$\RndThRH=0.5$]} 
    & \multicolumn{1}{c}{}\\   
   &       &       &      & \multicolumn{5}{c|}{}   & \multicolumn{6}{c|}{} & \multicolumn{1}{c}{}\\
   \multirow{2}{*}{$|\nodes|$} & \multirow{2}{*}{$|\arcs|$} & \textbf{\# RGUs}, & \multirow{2}{*}{$\nsample$} & \multirow{2}{*}{\textbf{T}} & \multirow{2}{*}{\textbf{ObjVal}} & \multirow{2}{*}{\textbf{ObjBnd}} & \multirow{2}{*}{\%\textbf{Gap}} & \multirow{2}{*}{\textbf{\#N}} & \multirow{2}{*}{\textbf{T}} & \multirow{2}{*}{\textbf{ObjVal}} & \multirow{2}{*}{\textbf{ObjBnd}} & \multirow{2}{*}{\%\textbf{Gap}} & \multirow{2}{*}{\textbf{\#N}} & \textbf{Time/} & \textbf{Avg Time} \\
        &       &  $|\nodeso|$   &   &       &       &       &  &   &       &       &       &       &  & \textbf{\RH Exec} & \textbf{RootRelax} \\ 
    \hline  &       &       &       &       &       &       &       &       &       &       &       &       &   &  & \\
    \multirow{4}[1]{*}{14} & \multirow{4}[1]{*}{40} & \multirow{4}[1]{*}{5} & \multicolumn{1}{c|}{10} & 600.02 &                2,401.74  &                2,424.37  & 0.942 &         886,768  & 600.04 &                2,401.74  &                2,430.48  & 1.197 &         643,161  &                        0.09  &                    0.04  \\
          &       &       & \multicolumn{1}{c|}{25} & 600.11 &                1,882.31  &                1,926.51  & 2.348 &         170,584  & 600.08 &                1,882.31  &                1,924.05  & 2.217 &         176,502  &                        0.19  &                    0.18  \\
          &       &       & \multicolumn{1}{c|}{50} & 600.13 &                1,888.39  &                1,929.74  & 2.190 &            75,227  & 600.13 &                1,888.39  &                1,929.73  & 2.189 &            71,228  &                        0.37  &                    0.57  \\
          &       &       & \multicolumn{1}{c|}{100} & 600.09 & \multicolumn{1}{c}{--}  &                1,986.56  &\multicolumn{1}{c}{--} &            32,395  & 600.24 &                1,911.18  &                1,952.88  & 2.182 &            25,211  &                        1.99  &                    4.83  \\
    \hline
          &       &       & \multicolumn{1}{c|}{} &       &       &       &       &       &       &       &       &       &       &       &  \\
    \multirow{4}[1]{*}{30} & \multirow{4}[1]{*}{82} & \multirow{4}[1]{*}{10} & \multicolumn{1}{c|}{10} & 600.29 &                2,584.47  &                2,621.01  & 1.414 &            52,667  & 600.17 &                2,584.47  &                2,621.09  & 1.417 &            54,055  &                        0.34  &                    0.51  \\
          &       &       & \multicolumn{1}{c|}{25} & 600.10 &                2,563.17  &                2,644.82  & 3.186 &            54,416  & 600.17 &                2,565.22  &                2,647.80  & 3.219 &            55,542  &                        0.67  &                    1.70  \\
          &       &       & \multicolumn{1}{c|}{50} & 600.12 &                2,485.95  &                2,654.72  & 6.789 &            16,216  & 600.19 &                2,574.43  &                2,658.70  & 3.273 &            22,621  &                        1.78  &                    7.00  \\
          &       &       & \multicolumn{1}{c|}{100} & 600.06 & \multicolumn{1}{c}{--}  &                2,216.28  &\multicolumn{1}{c}{--} &              5,236  & 600.12 &                2,109.87  &                2,216.07  & 5.033 &              3,469  &                        3.50  &                  36.54  \\
    \hline
          &       &       & \multicolumn{1}{c|}{} &       &       &       &       &       &       &       &       &       &       &       &  \\
    \multirow{4}[1]{*}{57} & \multirow{4}[1]{*}{156} & \multirow{4}[1]{*}{20} & \multicolumn{1}{c|}{10} & 268.95 &              11,161.15  &              11,162.27  & 0.010 &            28,300  & 51.81 &              11,161.15  &              11,161.15  & 0.000 &              1,771  &                        0.94  &                    1.94  \\
          &       &       & \multicolumn{1}{c|}{25} & 600.10 &                9,762.86  &              10,008.80  & 2.519 &              7,252  & 600.12 &                9,762.86  &              10,033.71  & 2.774 &              6,541  &                        2.21  &                  10.28  \\
          &       &       & \multicolumn{1}{c|}{50} & 600.13 & \multicolumn{1}{c}{--}  &              10,991.11  &\multicolumn{1}{c}{--} &              3,449  & 601.79 &              10,457.77  &              10,971.45  & 4.912 &              1,056  &                        4.12  &                  42.44  \\
          &       &       & \multicolumn{1}{c|}{100} & 603.21 & \multicolumn{1}{c}{--}  &              10,877.82  &\multicolumn{1}{c}{--} &                      1  & 600.16 &              10,245.95  &              10,891.10  & 6.297 &                      1  &                        7.14  &                204.73  \\
    \hline
          &       &       & \multicolumn{1}{c|}{} &       &       &       &       &       &       &       &       &       &       &       &  \\
    \multirow{4}[1]{*}{118} & \multirow{4}[1]{*}{358} & \multirow{4}[1]{*}{40} & \multicolumn{1}{c|}{10} & 29.69 &              22,798.57  &              22,799.46  & 0.004 &              1,232  & 34.80 &              22,798.57  &              22,800.57  & 0.009 &              1,046  &                        4.43  &                    6.38  \\
          &       &       & \multicolumn{1}{c|}{25} & 600.17 & \multicolumn{1}{c}{--}  &              23,448.94  &\multicolumn{1}{c}{--} &              3,543  & 600.35 &              23,391.43  &              23,444.93  & 0.229 &              7,731  &                        9.33  &                  20.91  \\
          &       &       & \multicolumn{1}{c|}{50} & 600.18 & \multicolumn{1}{c}{--}  &              24,080.32  &\multicolumn{1}{c}{--} &              3,160  & 600.20 &              24,013.95  &              24,079.44  & 0.273 &              2,675  &                      15.36  &                  87.26  \\
          &       &       & \multicolumn{1}{c|}{100} & 600.47 & \multicolumn{1}{c}{--}  &              24,038.75  &\multicolumn{1}{c}{--} &                  118  & 608.96 &              23,963.98  &              24,041.51  & 0.324 &                      1  &                      34.36  &                330.66  \\
    \hline
          &       &       & \multicolumn{1}{c|}{} &       &       &       &       &       &       &       &       &       &       &       &  \\
    \multirow{4}[0]{*}{300} & \multirow{4}[0]{*}{818} & \multirow{4}[0]{*}{80} & \multicolumn{1}{c|}{10} & 600.17 & \multicolumn{1}{c}{--}  &              52,263.83  &\multicolumn{1}{c}{--} &              7,438  & 600.22 &              51,404.37  &              52,181.32  & 1.511 &              4,292  &                        7.96  &                  29.43  \\
          &       &       & \multicolumn{1}{c|}{25} & 600.16 & \multicolumn{1}{c}{--}  &              53,047.86  &\multicolumn{1}{c}{--} &              1,368  & 600.21 &              52,074.98  &              53,020.28  & 1.815 &              1,095  &                      16.04  &                147.79  \\
          &       &       & 50    & 1200.10 & \multicolumn{1}{c}{--}  &              53,953.58  &\multicolumn{1}{c}{--} &                      1  & 1200.21 &              53,123.37  &              53,908.31  & 1.478 &                  129  &                      24.36  &                764.16  \\
          &       &       & 100   & 1200.61 & \multicolumn{1}{c}{--}  &           682,623.80  &\multicolumn{1}{c}{--} &                      1  & 1200.12 & \multicolumn{1}{c}{--}  &           682,623.80  &\multicolumn{1}{c}{--} &                      1  &  \multicolumn{1}{r|}{--}    &  TIME\_LIMIT \\
    \end{tabular}%
    }
    % \captionsetup{justification=centering}
    \caption{\scriptsize \textbf{{Duality-based} formulation} 
    with vs without Rounding Heuristic (\textbf{RH}) on IEEE test networks in Table~\ref{tab:IEEE-Dataset}.}
    \label{tab:RGU-B&BSummary-Results2}
\end{sidewaystable}%

%%%%%%%%%%%%%%%%%%%%%%%%%%%%%%%%%%%%%%%%%%%%%%%%%%%%%%%%%%%%%%%%%%%%
\section{Conclusion}
%%%%%%%%%%%%%%%%%%%%%%%%%%%%%%%%%%%%%%%%%%%%%%%%%%%%%%%%%%%%%%%%%%%%

We introduce a new Duality-based formulation for bilevel programs with spatial price equilibrium constraints. 
Together with the use of specially designed heuristics, this new bounded formulation allows the global solution of instances of EFL and RGUP orders of magnitude larger than is possible with the classical KKT-based formulation.  
% \vspace{-5mm}

\section*{Acknowledgement}
{This work was supported in part by AFOSR award FA9550-23-1-0451. 
The authors would like to thank the review team for their valuable feedback, which played a significant role in enhancing the quality of presentation of this paper.}

\section*{Data availability statement}
\url{https://t.ly/PC3D6} contains all the data used in numerical experiments (Section \ref{sec:NumericExp}).
\begin{itemize}
    \item \url{https://t.ly/kHwCP} for medium-sized networks in Section~\ref{sec:EFL-ComputResults1}, and \\
    \url{https://t.ly/QVBSX} for large-sized networks in Section~\ref{sec:EFL-ComputResults2}. 
    \begin{itemize}
        \item Links have folder for each (\#Nodes, \#Arcs) pair;
        \item Each folder contains data of different instances of a (\#Nodes, \#Arcs) pair;
        \item For each instance, separate \texttt{.csv} files specify upper \& lower-level parameters.
    \end{itemize}
    \item \url{https://t.ly/iP6iW} for IEEE bus networks in Section~\ref{sec:RGU-ComputResults1}.
    \newline The link has folders for each IEEE bus in Table \ref{tab:IEEE-Dataset}. Each folder contains
     \begin{itemize}
        \item A separate \texttt{.csv} file that specifies the upper-level parameters.
        \item A subfolder that has \texttt{.csv} files for lower-level parameters. 
        \item Separate \texttt{.csv} files for data of different number of uncertain samples in renewable generation.    
    \end{itemize}
\end{itemize}

\bibliography{sn-bibliography}

\clearpage

\begin{appendices}

%%%%%%%%%%%%%%%%%%%%%%%%%%%%%%%%%%%%%%%%%%%%%%%%%%%%%%%%%%%%%%%%%%%%
\section{EFL on networks} 
%%%%%%%%%%%%%%%%%%%%%%%%%%%%%%%%%%%%%%%%%%%%%%%%%%%%%%%%%%%%%%%%%%%%

%%%%%%%%%%%%%%%%%%%%%%%%%%%%%%%%%%
\subsection{Single-level reformulations}
\label{sec:EFL-Appendix_SingleLevel}
%%%%%%%%%%%%%%%%%%%%%%%%%%%%%%%%%%

Define $\idyf{} = \begin{pmatrix} \idyof{} \\ \idyif{} \end{pmatrix}$, 
$\idyD{} = \begin{pmatrix} \idyoD{} \\ \idyiD{} \end{pmatrix}$, 
$\idyS{} = \begin{pmatrix} \idyoS{} \\ \idyiS{} \end{pmatrix}$. 
To cast this formulation in the format of Section~\ref{sec:problem_setup}, one would choose 
$\bvay:=\begin{pmatrix} \bvaflo, & \bvadem,& \bvasup\end{pmatrix}$, 
$\ubvay:=(\bm{\infty},\bm{\infty},\bm{\infty})$, 
$\lhsy:=\begin{pmatrix} \idyf{} & \idyD{} & -\idyS{} \end{pmatrix}$, 
$\brhsB:=\begin{pmatrix} {0} \\ {0}\end{pmatrix}$
without using $\bvaw$ and $\lhsw$. 
The constraints of the single-level reformulations of EFL are:
\begin{subequations}
\begin{align}
     \label{constr:EFL-MPCC-Leader1} & {(\bvaz, \bvaq) \in \feas_\textrm{EFL}}, \\
     \label{constr:EFL-MPCC-Leader2} & {\bdvapi \geq 0}, \\
    & \label{constr:EFL-MPCC-PF0} \idyf{} \bvaflo + \idyD{} \bvadem - \idyS{} \bvasup - \begin{pmatrix}\bvaq \\ 0 \end{pmatrix} = 0, \\ 
   \label{constr:EFL-MPCC-PF} & \bvaflo \geq 0,\; \bvadem \geq 0,\; \bvasup \geq 0,\quad   \\
     & \bdvamuflo\geq0,\;\; \bdvamudem\geq0,\;\; \bdvamusup\geq0, \quad  \label{constr:EFL-MPCC-DF}\\     
     & \mybalpha(\bvaflo) + \idyf{\ltr} \bdvapi - \bdvamuflo = 0, \label{constr:EFL-MPCC-Stnry1} \\ 
     & - \mybbeta(\bvadem) + \idyD{\ltr} \bdvapi - \bdvamudem = 0, \\ 
     & \mybgamma(\bvasup) - \idyS{\ltr} \bdvapi - \bdvamusup = 0,       
     \label{constr:EFL-EqbmPrice}  \\
     & \bvaflo\tr \bdvamuflo  = 0,\;\; \bvadem\tr \bdvamudem= 0,\;\; \bvasup\tr \bdvamusup= 0 \label{constr:EFL-MPCC-CC}
\end{align}
\label{constr:EFL-MPCC}
\end{subequations}
where the set $\feas_\textrm{EFL}$ is specified in Section~\ref{sec:EFL-Descrp}.
The KKT-based reformulation of EFL is
\begin{align}
& \begin{aligned}
\optval{\kkt}{} = 
    \max\;\; & ({\bdvapio}-\coq)\tr \bvaq - \conewz\tr  \bvaz \\
    \st\;\; & \eqref{constr:EFL-MPCC-Leader1}-\eqref{constr:EFL-MPCC-CC}. %,\;\; \bdvapi \geq 0.
\end{aligned} \label{eq:EFL-SL-MPCC} \\ 
\intertext{The Duality-based reformulation obtained using Theorem~\ref{theorem:SL-Reform2} is}
& \begin{aligned}
\optval{\dual}{} =
    \max\;\; & - \left\langle \mybalpha(\bvaflo),\ \bvaflo \right\rangle + \left\langle \mybbeta(\bvadem),\ \bvadem \right\rangle - \left\langle \mybgamma(\bvasup),\ \bvasup \right\rangle -\coq\tr  \bvaq - \conewz\tr  \bvaz\\
    \st \;\; & \eqref{constr:EFL-MPCC-Leader1}-\eqref{constr:EFL-MPCC-CC}. %,\;\; \bdvapi \geq 0.
\end{aligned}   \label{eq:EFL-SL-LagrDual}
\end{align}

%%%%%%%%%%%%%%%%%%%%%%%%%%%%%%%%%%
\subsection{Rounding heuristic procedure}
\label{sec:EFL-Appendix_RH}
%%%%%%%%%%%%%%%%%%%%%%%%%%%%%%%%%%

The pseudo-code of this procedure is given in Algorithm~\ref{algo:roundingI}.

\algrenewcommand\algorithmicrequire{\textbf{Input:} }
\algrenewcommand\algorithmicensure{\textbf{Output:} }

\begin{algorithm}[!htbp]
\caption{Rounding Heuristic (\RH) for \eqref{eq:EFL-SL-LagrDual} } \label{algo:roundingI}
\begin{algorithmic}[1]
\Require{Relaxation solution $(\widehat \bvaz, \widehat \bvaq)$, rounding threshold $(\RndThRH \in [0,1])$}
\Ensure{Feasible solution $(\mymark \bvaz, \mymark \bvaq, \mymark \bvaflo, \mymark \bvadem, \mymark \bvasup, \mymark \bdvapi, \markbdvamuflo, \markbdvamudem,  \markbdvamusup)$ to \eqref{constr:EFL-MPCC} }
\vskip 1mm
\hrule
\vskip 1mm
\State Round $\widehat \bvaz$ into a binary vector  $\mymark \bvaz$ according to the rule
    \begin{align*}
        \mymark \bvaz_i = 
        \begin{cases}
        1 & \text{ if } \widehat \bvaz_i>\RndThRH \\
        0 & \text{ o.w.}
        \end{cases}
    \quad   \forall i\in\nodeso. 
    \end{align*}
\State
    Obtain a feasible vector $\mymark \bvaq$ that matches $\mymark \bvaz$ using the rule
    \begin{align*}
         \mymark \bvaq_i = 
        \begin{cases}
        \widehat \bvaq_i & \text{ if } \widetilde \bvaz_i=1 \\
        0 & \text{ o.w.}
        \end{cases}
     \quad   \forall i\in\nodeso.
    \end{align*}
\State
Fix $(\mymark \bvaz, \mymark \bvaq)$ and solve the lower-level primal and dual problems to obtain feasible  $(\mymark \bvaflo, \mymark \bvadem, \mymark \bvasup)$ and  $(\mymark \bdvapi, \markbdvamuflo, \markbdvamudem,  \markbdvamusup)$. 
\end{algorithmic} 
\end{algorithm}

%%%%%%%%%%%%%%%%%%%%%%%%%%%%%%%%%%
\subsection{Lower-level primal and dual problems}
\label{sec:EFL-Appendix_FollowerPrimalDual}
%%%%%%%%%%%%%%%%%%%%%%%%%%%%%%%%%%

We describe the lower-level primal and dual problems.
For the affine cost vector functions defined in Section~\ref{sec:EFL-InstGen}, 
\begin{alignat*}{5}
    & [\mybalpha(\bvaflo)]_{ij} = \myalpha{ij}{1}\cdot \bvaflo_{ij} + \myalpha{ij}{0}, &\quad& \forall (i,j)\in\arcs, \\
    & [\mybbeta(\bvadem)]_{i} = -\mybeta{i}{1}\cdot \bvadem_i + \mybeta{i}{0}, &\quad& \forall i\in\nodesD, \\
    & [\mybgamma(\bvasup)]_{j} = \mygamma{j}{1}\cdot \bvasup_j + \mygamma{j}{0}, &\quad& \forall j\in\nodesS,
\end{alignat*}
where $\myalpha{ij}{1}>0$, $\mybeta{i}{1}>0$, $\mygamma{j}{1}>0$, the follower's primal problem is the convex quadratic program:
\begin{subequations}\label{EFL:Follower_primal}
    \begin{align}
    \min_{\bvaflo,\bvadem,\bvasup}\;\; & 
    \sum_{(i,j)\in\arcs} \left(\frac{1}{2}\myalpha{ij}{1} \bvaflo_{ij}^2 + \myalpha{ij}{0} \bvaflo_{ij}\right) + 
    \sum_{i\in\nodesD} \left(\frac{1}{2}\mybeta{i}{1} \bvadem_i^2 - \mybeta{i}{0} \bvadem_i\right) + 
    \sum_{j\in\nodesS} \left(\frac{1}{2}\mygamma{j}{1} \bvasup_j^2 + \mygamma{j}{0} \bvasup_j\right)  \\
    \st\;\; & \idyf{} \bvaflo + \idyD{} \bvadem - \idyS{} \bvasup - \begin{pmatrix}\bvaq \\ 0 \end{pmatrix} = 0, \\
        & \bvaflo \geq 0,\; \bvadem \geq 0,\; \bvasup \geq 0.
\end{align}
\end{subequations}
Its dual problem is the concave quadratic program:

\begin{subequations}\label{EFL:Follower_dual}
    \begin{align}
    \max_{\bdvamuflo,\bdvamudem,\bdvamusup,\bdvapi}\;\; & -\frac{1}{2}\sum_{(i,j)\in\arcs} ({\bvaz}^{\bvaflo}_{ij})^2/\myalpha{ij}{1} -\frac{1}{2} \sum_{i\in\nodesD} \left({\bvaz}^{\bvadem}_i\right)^2/\mybeta{i}{1} -\frac{1}{2} \sum_{j\in\nodesS} \left({\bvaz}^{\bvasup}_j\right)^2/\mygamma{j}{1} - \bdvapiotr \bvaq  \\
    \st\;\; & \bvazflo=\mybalpha^0+\idyf{\ltr}\bdvapi-\bdvamuflo, \\
            & \bvazdem=-\mybbeta^0+\idyD{\ltr}\bdvapi-\bdvamudem, \\
            & \bvazsup=\mybgamma^0-\idyS{\tr}\bdvapi-\bdvamusup, \\
            & \bdvamuflo \geq 0,\; \bdvamudem \geq 0,\; \bdvamusup \geq 0.
\end{align}
\end{subequations}

In our computation, we solve a combined version of \eqref{EFL:Follower_primal} and \eqref{EFL:Follower_dual}, together with complementarity conditions and 
the requirement that  $\bdvapi\geq0$ to ensure that the obtained solution is feasible to \eqref{eq:EFL-SL-LagrDual}.

%%%%%%%%%%%%%%%%%%%%%%%%%%%%%%%%%%%%%%%%%%%%%%%%%%%%%%%%%%%%%%%%%%%%
\section{RGUP in power distribution networks}
%%%%%%%%%%%%%%%%%%%%%%%%%%%%%%%%%%%%%%%%%%%%%%%%%%%%%%%%%%%%%%%%%%%%

%%%%%%%%%%%%%%%%%%%%%%%%%%%%%%%%%%
\subsection{Single-level reformulations}
\label{sec:RGU-Appendix_SingleLevel}
%%%%%%%%%%%%%%%%%%%%%%%%%%%%%%%%%%
To cast this formulation in the format of Section~\ref{sec:problem_setup}, one would choose 
$\bvay:=\begin{pmatrix} \bvadem,& \bvasup\end{pmatrix}$, 
$\ubvay:=(\bm{\infty},\ubvasup)$, 
$\bvaw:= \bvaflo$, 
$\ubvaw:=\ubvaflo$, 
$\lhsy:=\begin{pmatrix} 
\idyoD{} & 0 \\ 
\idyiD{} & -\idyiS{} \\ 
0 & 0
\end{pmatrix}$, 
$\lhsw:=\begin{pmatrix} 
\idyof{} \\ 
\idyif{} \\ 
\reac 
\end{pmatrix}$, and 
$\brhsB:=\begin{pmatrix} {0} \\ {0} \\ {0} \end{pmatrix}$. 

Under Assumption~\ref{assump:Unique}, the following optimality conditions are necessary and sufficient for the problem faced by {the} ISO, where {
we use $\bdvath$ to represent dual variables on the capacity constraints, $\bdvamu$ to represent dual variables on non-negativity constraints, and where we specify the dual variables on equality constraints inside of square brackets:}   
\begin{subequations}
\begin{alignat}{5}    
\noalign{\textit{Primal Feasibility:}}         
         & \idyof{} \bvaflo + \idyoD{} \bvadem - \Diag(\samp) \bvaq = 0 \quad && [\bdvalao] \label{constr:RDG-MPCC-PFa}\\
         & \idyif{} \bvaflo + \idyiD{} \bvadem - \idyiS{} \bvasup = 0 \quad && [\bdvalai] \label{constr:RDG-MPCC-PFb}\\
         & \reac \bvaflo = 0 \quad && [\bdvaal] \label{constr:RDG-MPCC-PFc}\\
         & \bvasup \leq \ubvasup,\; \bvaflo \leq \ubvaflo \\
         & \bvaflo \geq 0,\; \bvadem \geq 0,\; \bvasup \geq 0 \quad  
     \label{constr:RDG-MPCC-PF} \\
\noalign{\textit{Dual Feasibility:}}     
     & \bdvamuflo\geq0,\;\; \bdvamudem\geq0,\;\; \bdvamusup\geq0,\;\; \bdvathflo\geq0,\;\; \bdvathsup \geq 0 \quad  \label{constr:RDG-MPCC-DF}\\
\noalign{\textit{Stationarity Conditions:}}     
         & \idyof{\ltr} \bdvalao + \idyif{\ltr} \bdvalai + \reac\tr\bdvaal + \bdvathflo - \bdvamuflo = 0 \quad  \label{constr:RDG-MPCC-Stnry1}\\
         & - \mybbeta(\bvadem) + \idyoD{\ltr} \bdvalao + \idyiD{\ltr} \bdvalai - \bdvamudem = 0 \quad  \label{constr:RDG-MPCC-Stnry2}\\
         & \mybgamma(\bvasup) - \idyiS{\ltr} \bdvalai + \bdvathsup - \bdvamusup = 0 \quad  \label{constr:RDG-MPCC-Stnry3}\\
\noalign{\textit{Complementarity Slackness:}}         
         & \bvaflo\tr \bdvamuflo  = 0,\;\; \bvadem\tr \bdvamudem= 0,\;\; \bvasup\tr \bdvamusup= 0 \quad  \\ %\label{constr:RDG-MPCC-CC}
         & (\ubvaflo-\bvaflo)\tr \bdvathflo  = 0,\;\; (\ubvasup-\bvasup)\tr \bdvathsup= 0. 
     \label{constr:RDG-EqbmPrice}  
\end{alignat}
\end{subequations}

Denote by $\Psi(\bvaq,\samp)=\left\{(\bvaflo,\bvadem,\bvasup,\bdvala,\bdvaal,\bdvamu,\bdvath): \eqref{constr:RDG-MPCC-PFa}-\eqref{constr:RDG-EqbmPrice}\right\}$. 
The single-level reformulation for optimally locating RGUs is a two-stage stochastic program maximizing the leader firm's expected profit given by:
\begin{subequations}\label{eq:RGU-SL-TwoStageSP}
\begin{align}
\max_{\substack{\bvaz,\bvaq}}\;\; & - \conewz\tr  \bvaz - \coq\tr \bvaq + \expec\Big[\max_{\substack{(\bvaflo,\bvadem,\bvasup,\bdvala,\bdvaal,\bdvamu,\bdvath)\\ \in\ \Psi(\bvaq,\samp)}} \bdvalaotr \Diag(\samp) \bvaq\Big]  \\
\st\;\; & (\bvaz, \bvaq) \in \feas_\textrm{RGUP}, \label{constr:RDG-MPCC-Leader3} \\
        & \bdvala \geq 0, \label{constr:RDG-MPCC-Leader2} 
\end{align}
\end{subequations}
where {the set $\feas_\textrm{RGUP}$ is specified in Section \ref{sec:RGU-Descrp}}, $\conewz_i$ is the fixed cost for installing a RGU at node $i\in\nodeso$ and 
$\coq_i$ is the cost per unit of RGU capacity installation at node $i\in\nodeso$. 

Assume next that we draw finite samples for the uncertainty $\{\samp^{\langle {k}\rangle}\}_{{k}=1}^{\nsample}$. 
Then using sample average approximation for the expectation, the single-level reformulation becomes
\begin{align}\label{eq:RGU-SL-MPCC}
    \begin{aligned}
    \max_{\substack{\bvaz,\bvaq,\bvaflo^{\langle {k}\rangle},\bvadem^{\langle {k}\rangle},\bvasup^{\langle {k}\rangle},\\ \bdvala^{\langle {k}\rangle},\bdvaal^{\langle {k}\rangle},\bdvamu^{\langle {k}\rangle},\bdvath^{\langle {k}\rangle}}}\;\; 
    &  \frac{1}{\nsample} \sum_{{k}=1}^{\nsample} {\bm{\lambda}_{0}^{\langle {k}\rangle}}\tr \Diag(\samp^{\langle {k}\rangle}) \bvaq  - \conewz\tr  \bvaz - \coq\tr \bvaq \\
    \st\;\; \qquad \quad &  \eqref{constr:RDG-MPCC-Leader3},\\
    & (\bvaflo^{\langle {k}\rangle},\bvadem^{\langle {k}\rangle},\bvasup^{\langle {k}\rangle},\bdvala^{\langle {k}\rangle},\bdvaal^{\langle {k}\rangle},\bdvamu^{\langle {k}\rangle},\bdvath^{\langle {k}\rangle}) \in \Psi(\bvaq,\samp^{\langle {k}\rangle}), \forall {k} \in [\nsample], \\
    & \bdvala^{\langle {k}\rangle}\geq0, \forall {k} \in [\nsample].
    \end{aligned}
\end{align}
Using Theorem~\ref{theorem:SL-Reform2} for each sample of uncertainty, the objective function in \eqref{eq:RGU-SL-MPCC} can be re-expressed to obtain
\begin{align}\label{eq:RGU-SL-LagrDual}
    \begin{aligned}
    \max_{\substack{\bvaz,\bvaq,\bvaflo^{\langle {k}\rangle},\\ \bvadem^{\langle {k}\rangle},\bvasup^{\langle {k}\rangle}, \bdvala^{\langle {k}\rangle}, \\ \bdvaal^{\langle {k}\rangle},\bdvamu^{\langle {k}\rangle},\bdvath^{\langle {k}\rangle}}}\; 
    &
    \begin{array}{l}
    \frac{1}{\nsample} \sum_{{k}=1}^{\nsample} \left(\left\langle \mybbeta(\bvadem^{\langle {k}\rangle}),\ \bvadem^{\langle {k}\rangle} \right\rangle - \left\langle \mybgamma(\bvasup^{\langle {k}\rangle}),\ \bvasup^{\langle {k}\rangle} \right\rangle -{\ubvaflotr \bdvathflo^{\langle {k}\rangle}} -{\ubvasuptr \bdvathsup^{\langle {k}\rangle}}\right)  \\
    \qquad - \conewz\tr  \bvaz - \coq\tr \bvaq 
    \end{array}
    \\
        \st\;\; \qquad &  \eqref{constr:RDG-MPCC-Leader3}, \\
        & (\bvaflo^{\langle {k}\rangle},\bvadem^{\langle {k}\rangle},\bvasup^{\langle {k}\rangle},\bdvala^{\langle {k}\rangle},\bdvaal^{\langle {k}\rangle},\bdvamu^{\langle {k}\rangle},\bdvath^{\langle {k}\rangle}) \in \Psi(\bvaq,\samp^{\langle {k}\rangle}), \forall {k} \in [\nsample]\\ & \bdvala^{\langle {k}\rangle}\geq0, \forall {k} \in [\nsample].
    \end{aligned}
\end{align}

%%%%%%%%%%%%%%%%%%%%%%%%%%%%%%%%%%
\subsection{Rounding heuristic procedure}
\label{sec:RGU-Appendix_RH}
%%%%%%%%%%%%%%%%%%%%%%%%%%%%%%%%%%

The Rounding Heuristic (\RH) procedure used here is essentially that given in {Algorithm~\ref{algo:roundingI} in Appendix~\ref{sec:EFL-Appendix_RH}} except that Step 3 is replaced with the following:
\begin{itemize}
\item STEP-3: Fix $(\widetilde \bvaz, \widetilde \bvaq)$ and solve the lower-level primal and dual problems for each uncertainty sample $\samp^{\langle {k}\rangle}$ to obtain feasible $({\mymark \bvaflo}^{\langle {k}\rangle}, {\mymark \bvadem}^{\langle {k}\rangle}, {\mymark \bvasup}^{\langle {k}\rangle})$ and  $(\markmyblambda^{\langle {k}\rangle}, \widetilde{\bm{\alpha}}^{\langle {k}\rangle}, \markbdvamuflo^{\langle {k}\rangle},\markbdvamudem^{\langle {k}\rangle}, \markbdvamusup^{\langle {k}\rangle}, \markbdvathflo^{\langle {k}\rangle}, \markbdvathsup^{\langle {k}\rangle})$. 
Due to the lower-level cost vector in our instances being affine, the primal and dual problems are convex quadratic programs \eqref{RGU:Follower_primal} and \eqref{RGU:Follower_dual}, which are described in Appendix~\ref{sec:RGU-Appendix_FollowerPrimalDual} and can be efficiently solved using GUROBI.
\end{itemize}
%\sout{Similar to Section~\ref{sec:EFL-ComputResults2}, the above \RH procedure is also employed with a probability $\probRH$.} 
We remark that $2\nsample$ convex quadratic programs must be solved to find a feasible solution,  where $\nsample$ is the sample size of uncertainty.

%%%%%%%%%%%%%%%%%%%%%%%%%%%%%%%%%%
\subsection{Lower-level primal and dual problems}
\label{sec:RGU-Appendix_FollowerPrimalDual}
%%%%%%%%%%%%%%%%%%%%%%%%%%%%%%%%%%

Assume that the affine cost vector functions are defined as in Section~\ref{sec:RGU-InstGen}:
\begin{alignat*}{5}
    & [\mybbeta(\bvadem)]_{i} = -\mybeta{i}{1}\cdot \bvadem_i + \mybeta{i}{0}, &\quad& \forall i\in\nodesD \\
    & [\mybgamma(\bvasup)]_{j} = \mygamma{j}{1}\cdot \bvasup_j + \mygamma{j}{0}, &\quad& \forall j\in\nodesS
\end{alignat*}
where $\mybeta{i}{1}>0$ and $\mygamma{j}{1}>0$. 
Then, for each sample of uncertainty $\samp^{\langle {k}\rangle}$, the follower's primal problem is the convex quadratic program:
\begin{subequations}\label{RGU:Follower_primal}
    \begin{align}
    \min_{\bvaflo^{\langle {k}\rangle},\bvadem^{\langle {k}\rangle},\bvasup^{\langle {k}\rangle}}\;\; &  \sum_{i\in\nodesD} \left(\frac{1}{2}\mybeta{i}{1} (\bvadem_i^{\langle {k}\rangle})^2 - \mybeta{i}{0} \bvadem_i^{\langle {k}\rangle}\right) + \sum_{j\in\nodesS} \left(\frac{1}{2}\mygamma{j}{1} (\bvasup_j^{\langle {k}\rangle})^2 + \mygamma{j}{0} \bvasup_j^{\langle {k}\rangle}\right)  \\
    \st\;\;  
    & \idyof{} \bvaflo^{\langle {k}\rangle} + \idyoD{} \bvadem^{\langle {k}\rangle} - \Diag(\samp{^{\langle {k}\rangle}}) \bvaq = 0 \\
         & \idyif{} \bvaflo^{\langle {k}\rangle} + \idyiD{} \bvadem^{\langle {k}\rangle} - \idyiS{} \bvasup^{\langle {k}\rangle} = 0 \\
         & \reac \bvaflo^{\langle {k}\rangle} = 0 \\
         & 0 \leq \bvaflo^{\langle {k}\rangle} \leq \ubvaflo,\;\; \bvadem^{\langle {k}\rangle} \geq 0,\;\;  0 \leq \bvasup^{\langle {k}\rangle} \leq \ubvasup.
\end{align}
\end{subequations}
Its dual problem is a concave quadratic program:
\begin{subequations}\label{RGU:Follower_dual}
    \begin{align}
    \max_{\substack{\bdvala^{\langle {k}\rangle},\bdvaal^{\langle {k}\rangle},\bdvamuflo^{\langle {k}\rangle},\bdvamudem^{\langle {k}\rangle},\\ \bdvamusup^{\langle {k}\rangle},\bdvathflo^{\langle {k}\rangle},\bdvathsup^{\langle {k}\rangle}}}\;\; & 
    \begin{array}{l}
    -\frac{1}{2} \sum_{i\in\nodesD} \left({\bvaz_i^{\bvadem}}^{\langle {k}\rangle}\right)^2/\mybeta{i}{1} 
    -\frac{1}{2} \sum_{j\in\nodesS} \left({\bvaz_j^{\bvasup}}^{\langle {k}\rangle}\right)^2/\mygamma{j}{1}  \\
    \qquad - \ubvaflotr \bdvathflo^{\langle {k}\rangle} 
    - \ubvasuptr \bdvathsup^{\langle {k}\rangle} 
    - \bdvalaotr^{\langle {k}\rangle} \Diag(\samp{^{\langle {k}\rangle}})\bvaq
    \end{array}
    \\
    \st\;\; & \idyof{\ltr} \bdvalao^{\langle {k}\rangle} + \idyif{\ltr} \bdvalai^{\langle {k}\rangle} + \reac\tr\bdvaal^{\langle {k}\rangle} + \bdvathflo^{\langle {k}\rangle} - \bdvamuflo^{\langle {k}                   \rangle} = 0 \\
            & \bvazdem^{\langle {k}\rangle}=-\mybbeta^0+\idyoD{\ltr}  \bdvalao^{\langle {k}\rangle}+\idyiD{\ltr} \bdvalai^{\langle {k}\rangle}-\bdvamudem^{\langle {k}\rangle} \\
            & \bvazsup^{\langle {k}\rangle}=\mybgamma^0-\idyiS{\ltr} \bdvalai^{\langle {k}\rangle}+\bdvathsup^{\langle {k}\rangle}-\bdvamusup^{\langle {k}\rangle} \\
            & \bdvamuflo^{\langle {k}\rangle} \geq 0,\; \bdvamudem^{\langle {k}\rangle} \geq 0,\; \bdvamusup^{\langle {k}\rangle} \geq 0,\\
            & \bdvathflo^{\langle {k}\rangle} \geq 0,\; \bdvathsup^{\langle {k}\rangle} \geq 0.  \nonumber
\end{align}
\end{subequations}

In our computation, we solve a combined version of \eqref{RGU:Follower_primal} and \eqref{RGU:Follower_dual}, together with complementarity conditions and 
the requirement that  $\bdvala^{\langle {k}\rangle}\geq0$ to ensure that the obtained solution is feasible to \eqref{eq:RGU-SL-LagrDual}.

\end{appendices}

\end{document}